\documentclass[10pt,reqno]{amsart}

\usepackage{enumerate}
\usepackage{amsmath, amssymb, amsthm}
\usepackage{mathrsfs}
\usepackage{esint}
\usepackage{xcolor}
\usepackage{mathtools}
\usepackage{hyperref}
\usepackage{bm}
\usepackage{accents}
\usepackage{amsaddr}

\makeatletter

\numberwithin{equation}{section}

\newtheorem{thm}{Theorem}[section]
\newtheorem{theorem}[thm]{Theorem}
\newtheorem{lemma}[thm]{Lemma}
\newtheorem{corollary}[thm]{Corollary}

\newtheorem{prop}[thm]{Proposition}
\theoremstyle{definition}

\newtheorem{definition}[thm]{Definition}

\theoremstyle{remark}

\newtheorem{remark}[thm]{\bf{Remark}}

\newcommand\aint{-\hspace{-0.38cm}\int}

\newcommand\bE{\mathbb{E}}

\newcommand\bH{\mathbb{H}}

\newcommand\bL{\mathbb{L}}
\newcommand\bM{\mathbb{M}}
\newcommand\bN{\mathbb{N}}

\newcommand\bP{\mathbb{P}}

\newcommand\bR{\mathbb{R}}

\newcommand\bZ{\mathbb{Z}}

\newcommand\cD{\mathcal{D}}

\newcommand\cF{\mathcal{F}}
\newcommand\cG{\mathcal{G}}
\newcommand\cH{\mathcal{H}}

\newcommand\cT{\mathcal{T}}

\newcommand\frH{\mathfrak{H}}

\newcommand{\mysection}[1]{\section{#1}}

\begin{document}

\title[PDEs with fractional Laplacian on $C^{1,1}$ open sets]{Sobolev regularity theory for the non-local elliptic and parabolic equations on $C^{1,1}$ open sets}

\author{Jae-Hwan Choi$^{1,2}$, Kyeong-Hun Kim$^1$, and Junhee Ryu$^{*1}$}

\email{jaehwanchoi@kaist.ac.kr}

\email{kyeonghun@korea.ac.kr}
\thanks{}

\email{junhryu@korea.ac.kr}

\address{$^1$Department of Mathematics, Korea University, 145 Anam-ro, Seongbuk-gu, Seoul, 02841, Republic of Korea}

\address{$^2$(Current address) Department of Mathematical Sciences, KAIST, 291, Daehak-ro, Yuseong-gu, Daejeon 34141, Republic of Korea}

\thanks{The authors were supported by the National Research Foundation of Korea(NRF) grant funded by the Korea government(MSIT) (No. NRF-2019R1A5A1028324)}

\thanks{$^*$Corresponding author: Junhee Ryu}



\begin{abstract}
We study the zero exterior problem for the elliptic equation
$$
\Delta^{\alpha/2}u-\lambda u=f, \quad  x\in D\,; \quad u|_{D^c}=0
$$
as well as for the parabolic equation
$$
u_t=\Delta^{\alpha/2}u+f, \quad t>0,\, x\in D \,; \quad u(0,\cdot)|_D=u_0, \,u|_{[0,T]\times D^c}=0.
$$
Here, $\alpha\in (0,2)$, $\lambda \geq 0$ and $D$ is a $C^{1,1}$ open set.
We prove uniqueness and existence of solutions in weighted Sobolev spaces, and obtain global Sobolev and H\"older estimates of solutions and their arbitrary order derivatives. We measure the Sobolev and H\"older regularities of solutions and their arbitrary derivatives using a system of weights consisting of appropriate powers of the distance to the boundary. 
The range of admissible powers of the distance to the boundary is sharp.
\end{abstract}

\maketitle

\mysection{Introduction}

We study the elliptic equation
\begin{equation} \label{elliptic eqn}
\begin{cases}
\Delta^{\alpha/2}u(x)-\lambda u(x)=f(x),\quad &x\in D,
\\
u(x)=0,\quad &x\in D^c,
\end{cases}
\end{equation}
and the parabolic equation
\begin{equation} \label{parabolic eqn}
\begin{cases}
\partial_t u(t,x)=\Delta^{\alpha/2}u(t,x)+f(t,x),\quad &(t,x)\in(0,T)\times D,
\\
u(0,x)=u_0(x),\quad & x\in D,
\\
u(t,x)=0,\quad &(t,x)\in [0,T]\times D^c,
\end{cases}
\end{equation}
where  $\alpha\in (0,2)$ and $D$ is either a half space or a bounded $C^{1,1}$ open set.  The fractional Laplacian $\Delta^{\alpha/2}u$ is defined as 
$$
\Delta^{\alpha/2} u(x) :=c_d\lim_{\varepsilon \downarrow 0}\int_{|y|>\varepsilon} \frac{u(x+y)-u(x)}{|y|^{d+\alpha}}dy, \qquad (c_d:=\frac{2^{\alpha}\Gamma(\frac{d+\alpha}{2})}{\pi^{d/2}|\Gamma(-\alpha/2)|}).
$$

In the probabilistic point of view, equations \eqref{parabolic eqn} and \eqref{elliptic eqn}  are related to a certain pure-jump process which is forced to assume undefined or killed state  when it leaves the open set $D$. The zero exterior condition describes that the influence of the jump process vanishes or is ignored when the process is outside of $D$. See Section \ref{Main results} for detail. In fact, the equations are ill-posed if only zero-boundary condition is assigned.

In this article we study equations \eqref{elliptic eqn} and \eqref{parabolic eqn} in the weighted Sobolev spaces $H^{\gamma}_{p,\theta}(D)$ and $L_p((0,T); H^{\gamma}_{p,\theta}(D))$, respectively.  Here  $p>1$ and  $\theta, \gamma \in \bR$.  For instance, if $\gamma=0,1,2,\cdots$, then
$$
\|u\|_{H^{\gamma}_{p,\theta}(D)}=\left(\sum_{k=0}^{\gamma} \int_D |\rho^{k} D^{k}u|^p \rho^{\theta-d} dx \right)^{1/p},
$$
where $\rho(x):=dist(x,\partial D)$. In general, we use  a unified way to define the spaces  $H^{\gamma}_{p,\theta}(D)$ for all  $\gamma\in \bR$. The powers of $\rho$ are used to control  the behaviors of functions  near   the boundary.

The main contribution of this article is to present weighted Sobolev regularity of arbitrary nonnegative real order derivatives of solutions.   To be more precise,  for elliptic equation \eqref{elliptic eqn} we prove  for any  $\gamma \geq 0$ and $p>1$,
\begin{equation}
 \label{eqn 4.11.1}
\|u\|_{H_{p,\theta-\alpha p/2}^{\gamma+\alpha}(D)} \leq C \|f\|_{H^{\gamma}_{p,\theta+\alpha p/2}(D)}\end{equation}
provided that $\theta \in (d-1, d-1+p)$.  The admissible range of $\theta$ is sharp (cf. Remark \ref{remark sharp}). We also prove a parabolic version of \eqref{eqn 4.11.1} for parabolic equation   \eqref{parabolic eqn}.  See Theorems \ref{21.10.18.2056} and \ref{21.10.18.2057} for our full Sobolev regularity results of the elliptic and parabolic equations. In particular,  if $\gamma=0$ then   \eqref{eqn 4.11.1} implies
  \begin{equation}
  \label{eqn 9.3.1}
 \int_{D}(|\rho^{-\alpha/2}u|^p+ |\rho^{\alpha/2}\Delta^{\alpha/2}u|^p) \rho^{\theta-d} dx
\leq C  \int_{D} |\rho^{\alpha/2}f|^p  \rho^{\theta-d} dx.
\end{equation}

Note that due to the presence of $\rho^{\alpha/2}$ beside $f$ in \eqref{eqn 9.3.1}, the function $f$ is  allowed to blow up near the boundary of $D$. Indeed, it can behave like $\rho^{-\alpha/2}$ near $\partial D$.

We also  obtain global space-time H\"older estimates of  arbitrary derivatives of solutions (see Corollaries \ref{21.09.22.1621} and \ref{cor ellipitic}).   One advantage of our results is that it gives H\"older estimates of solutions even when the free terms are quite rough. 
For instance, if 
$\alpha-\frac{d}{p} \geq \delta\in(0,1)$, then for the elliptic equation we prove
    \begin{equation}
    \label{ho}
  |\rho^{\frac{\theta}{p}-\frac{\alpha}{2}}u|_{C(D)}+|\rho^{\delta+\frac{\theta}{p}-\frac{\alpha}{2}}u|_{C^{\delta}(D)}\leq C \|\psi^{\alpha/2}f\|_{L_{p,\theta}(D)}. 
  \end{equation}

Now we give a description on the mostly related works below. Our focus lies in the results on domains. Accordingly, regarding  the results on the whole space $\bR^d$, we only refer  e.g. to  \cite{bae2015schauder,bass2009regularity,dong2013schauder,kim2015holder,kuhn2019schauder} for H\"older estimates and \cite{dong2012lp,kim2013parabolic,kim2016lp,kim2019lp,mikulevivcius2017p,mikulevivcius2019cauchy} for $L_p$ estimates. 

First, we describe H\"oder estimates.  As for elliptic equation \eqref{elliptic eqn}, it was proved  in \cite{ros2014dirichlet} that  
\begin{equation*}
    f\in L_{\infty}(D)\quad\implies \quad u\in C^{\alpha/2}(\bR^d),\,\,\rho^{-\alpha/2}u\in C^s(D)
\end{equation*}
for some $s>0$.  Here, $\rho(x):=dist (x, \partial D)$.  Higher order  estimate
$$
|u|^{(-\alpha/2)}_{\beta+\alpha;D}\leq C \left(|u|_{C^{\alpha/2}(\bR^d)}+|f|^{(\alpha/2)}_{\beta;D} \right), \qquad  \beta>0
$$
was also obtained in \cite{ros2014dirichlet}, where $|\cdot|^{(a)}_{b}$ denotes the interior H\"older norm (see e.g. \cite{GT} or \cite{KK2004}).  The result of \cite{ros2014dirichlet} was generalized for  elliptic equations with stable-like operators in \cite{arapostathis2016dirichlet,kim2019boundary,ros2016regularity}.  We also refer to  \cite{kuhn2021interior,ros2016dirichlet} for the local result of the type 
\begin{equation}
\nonumber
\label{21.10.28.20.54}
 f\in C^{\beta}(D) \quad \implies u\in C^{\beta+\alpha}_{loc}(D), \quad \beta>0
\end{equation}
 proved for non-local elliptic equations with singular kernels or general operators.  Also, see \cite{caffarelli2009regularity,kim2021generalized,ros2016boundary} for  related works on non-linear elliptic equations. Now, we discuss the results on  parabolic equation \eqref{parabolic eqn}.   In \cite{fernandez2017regularity}, it was proved that  if $u_0\in L_2(D)$ and $ f\in L_{\infty}((0,T)\times D)$, then
  \begin{equation}
  \label{pa111}
 u\in C_{t,x}^{1-\varepsilon,\alpha/2}((t_0,T)\times D),\,\rho^{-\alpha/2}u\in C_{t,x}^{\frac{1}{2}-\frac{\varepsilon}{\alpha},\frac{\alpha}{2}-\varepsilon}((t_0,T)\times D)
\end{equation}
for any $\varepsilon>0$ and $t_0\in(0,T)$.  Note that this result is local with respect to the time variable. 
For a global estimate, we refer to  \cite {zhang2018dirichlet}, which  in particular proved 
\begin{equation}
\label{pa3}
\sup_{t\leq T} \left(|u|^{(-\theta)}_{\alpha+\gamma;D}+ |\Delta^{\alpha/2}u|^{(\alpha-\theta)}_{\gamma;D}\right) \leq C \left(|u_0|^{(-\theta)}_{\alpha+\gamma;D}+\sup_{t\leq T}|f|^{(\alpha-\theta)}_{\gamma;D}\right)
\end{equation}
for any $\theta\in(0,\alpha/2)$ and $\gamma\in(0,1)$. This estimate does not give H\"older regularity with respect to the time variable. As compared to \eqref{pa111} and \eqref{pa3},  our results give global H\"older regularity with respect to both time and spacial variables.

Next, we describe results in $L_p$ spaces.
The global summability results were studied e.g. in \cite{Abdel Global,leonori2015basic}.
For instance, for elliptic equation \eqref{elliptic eqn}, the inequality
\begin{equation}
\label{eqn 7.11}
\|u\|_{L_{\frac{d p}{d-\alpha p}}(D) }+\|\Delta^{\alpha/4}u\|_{ L_{\frac{d p}{d-\alpha p/2}}(D) } \leq C \|f\|_{L_p (D)}, \quad (1<p<2d/(d+\alpha))
\end{equation}
was proved in \cite{leonori2015basic}.
We remark that \eqref{eqn 7.11} does not cover the full regularity of solution, that is, estimate of  $\Delta^{\alpha/2}u$ is not covered.
There are also several interior regularity results, such as those introduced in \cite{zuazua2017,biccari2018local,cozzi2017interior,nowak2020hs}. 
For instance, the results
$$
f\in L_p(D) \implies \quad u\in H^{\alpha}_{p,loc}(D), \quad (1<p<\infty)
$$
and
\begin{equation}
\label{eqn 4.23.4}
f\in L_{p_*}(D)  \implies\quad u\in H^{\alpha/2}_{p,loc}(D)  \quad \quad (p>2, \, p_*:=\max\{\frac{pd}{d+p\alpha/2}, 2\})
\end{equation}
were proved in  \cite{zuazua2017} and \cite{nowak2020hs} respectively. Note that $p_*<p$ since $p>2$. 
We also refer to \cite{bogdan2020extension,felsinger2015dirichlet,hoh1996dirichlet} for results on Hilbert spaces.  We finally refer to \cite{grubb2014local,grubb2018regularity} for the regularity  results in the $\mu$-transmission spaces $H_p^{\mu(s)}(D)$, proved for the equations  with pseudo-differential operators satisfying the $\mu$-transmission property.
In particular, it is proved that if $f\in L_p(D)$, then elliptic problem \eqref{elliptic eqn} has a unique solution $u\in H_p^{\alpha/2(\alpha)}(D)$.

Our approach is different from those in the above mentioned articles and is based on weighted Sobolev spaces.  We summarize our results and their differences from above mentioned results as follow.

\begin{itemize}
\item{}  We prove the weighted Sobolev regularity result of  solutions to both elliptic and parabolic equations. For instance, for the elliptic equation we prove  \eqref{eqn 4.11.1}, which is 
\begin{equation}
\label{eqn 4.23}
\|u\|_{H_{p,\theta-\alpha p/2}^{\gamma+\alpha}(D)} \leq C \|f\|_{H^{\gamma}_{p,\theta+\alpha p/2}(D)}, 
 \end{equation}
for any $p>1,  \gamma \geq 0$ and  $\theta\in (d-1, d-1+p)$. 

\item{} We emphasize that unlike in any of above mentioned articles,  our results, e.g. \eqref{eqn 4.23},  are proved for any $\gamma\geq 0$, despite that $D$ is an only $C^{1,1}$ open set. This is possible because we use appropriate weighted Sobolev spaces.

\item{} As can be seen from \eqref{eqn 9.3.1}, regularity of $\rho^{-\alpha/2}u$ near $\partial D$ changes according to $\theta$.  For instance, as  $\theta \downarrow d-1$, $\rho^{-\alpha/2}u$   decays faster.  Among the results described above, the closest result to \eqref{eqn 9.3.1}(or \eqref{eqn 4.23})  can be found in \cite{grubb2014local,grubb2018regularity}, which in particular show that  if $D$ is a bounded $C^{\infty}$ domain and $f\in L_p(D)$ then $\rho^{-\alpha/2}u\in H^{\alpha/2}_p(D)$.  This result  is close to \eqref{eqn 9.3.1} if $\theta=d$.  Besides that $\theta$ can vary in the present article, we require the weaker assumption 
$\rho^{\alpha/2}f \in L_p(D, \rho^{\theta-d}dx)$. 

\item{} We use \eqref{eqn 4.23} to prove versions of \eqref{eqn 7.11} and \eqref{eqn 4.23.4} under weaker assumption on $f$. See Remark \ref{23.04.14.21} for a comparison of our results with those in  \cite{Abdel Global,leonori2015basic,nowak2020hs}.

\item{}   We also obtain  global H\"older estimates  for the elliptic and parabolic equations; the free terms can be quite irregular and unbounded. For instance, in \eqref{ho} we only require $\rho^{\alpha/2}f\in L_{p,\theta}(D)$. 

\end{itemize}

Now, we introduce the organization of this article.  In Section \ref{Main results}, we introduce our main results, Sobolev space theory and H\"older estimates of solutions. In Section \ref{zero order}, we study the representation of solutions and estimate the zero-th order derivative of solutions.  In Section \ref{21.09.20.1430}, we prove higher regularity of solutions, and we give the proofs of main results  in Section \ref{subproof}.

We finish the introduction with notations used in this article. We use $``:="$ or $``=:"$ to denote a definition.  $\bN$ and $\bZ$ denote the natural number system and the integer number system, respectively.  We denote $\bN_+:=\bN\cup\{0\}$, and  as usual $\bR^d$ stands for the Euclidean space of points $x=(x^1,\dots,x^d)$,
$$
B_r(x)=\{y\in\bR^d : |x-y|<r\}, \quad \bR^{d}_+=\{(x^1,\dots,x^d)\in\bR : x^1>0\}.
$$
For nonnegative functions $f$ and $g$, we write $f(x)\approx g(x)$  if there exists a constant $C>0$, independent of $x$,  such that $C^{-1} f(x)\leq g(x) \leq C f(x)$. For  multi-indices $\beta=(\beta_1,\cdots,\beta_d)$, $\beta_i\in\bN_+$, and functions $u(x)$ depending on $x$,
$$
 D_iu(x):=\frac{\partial u}{\partial x^i},\quad  D^{\beta}_xu(x):=D^{\beta_d}_{d}\cdots D_1^{\beta_1}u(x).
$$
We also use  $D^n_xu$ to denote   the  partial derivatives of order $n\in\bN_+$ with respect to the space variables. For an open set $U\subset \bR^d$, $C(U)$ denotes the space of continuous functions $u$ in $U$ such that $|u|_{C(U)}:=\sup_U |u(x)|<\infty$. $C_0(U)$ is the set of functions in  $C(U)$  satisfying $\underset{|x|\to\infty}{\lim} u(x)=0$ and $\underset{x\to \partial U}{\lim} u(x)=0$.  By $C^2_b(U)$ we denote the space of functions  whose derivatives of order up  to $2$ are in $C(U)$.   For an open set $V\subset \bR^m$, where $m\in \bN$,   by $C_c^\infty(V)$ we denote the space of infinitely differentiable functions with compact support in $V$. For  a Banach space $F$ and $\delta\in (0,1]$,   $C^{\delta}(V;F)$ denotes the space of $F$-valued continuous functions $u$ on $V$  such that
\begin{eqnarray*}
    |u|_{C^{\delta}(V;F)}&:=&|u|_{C(V;F)}+[u]_{C^{\delta}(V;F)}
    \\
    &:=&
     \sup_{x\in V}|u(x)|_F+\sup_{x,y\in V}\frac{|u(x)-u(y)|_F}{|x-y|^{\delta}}<\infty.
\end{eqnarray*}
Also, for  $p>1$ and a measure $\mu$ on $V$, $L_p(V, \mu; F)$  denotes the set of $F$-valued Lebesgue measurable functions $u$ such that 
$$
\|u\|_{L_p(V, \mu; F)}:=\left(\int_V |u|^p_F \,d\mu\right)^{1/p}<\infty.
$$
 We drop $F$ and $\mu$ if $F=\bR$ and $\mu$ is the Lebesgue  measure.  By  $\cD(U)$, where $U$ is an open set in $\bR^d$, 
we denote the space of all distributions on $U$, and  for given $f\in \cD(U)$, the action of $f$ on $\phi \in C_c^\infty(U)$ is denoted by
$$
( f, \phi)_{U} :=f(\phi).
$$
Finally, if  we write $C=C(a,b,\cdots)$, then this means that the constant $C$ depends only on $a,b,\cdots$.

\mysection{Main results} \label{Main results}

In subsection 2.1 we prove the uniqueness and existence results in a wide class of function spaces, and  we give the regularity of solutions  in subsection 2.2. 

Throughout this article,  $D$ is either a half space $\bR^d_+$ or a bounded $C^{1,1}$ open set.

\subsection{Uniqueness and existence} \label{notion of weak}

For   suitable functions $f$ defined on $\bR^d$ (e.g. $f\in C^2_b(\bR^d)$),  we define the fractional Laplacian $\Delta^{\alpha/2}f$ as
\begin{align} \label{21.07.20.1710}
\Delta^{\alpha/2} f(x) :=c_d \lim_{\varepsilon \downarrow 0}\int_{|y|>\varepsilon} \frac{f(x+y)-f(x)}{|y|^{d+\alpha}}dy,
\end{align}
where $c_d=\frac{2^{\alpha}\Gamma(\frac{d+\alpha}{2})}{\pi^{d/2}|\Gamma(-\alpha/2)|}$.   Also, for functions $f,g$ defined on  $E\subset \bR^d$, we set
$$
\langle f,g \rangle_E:=\int_E fg \,dx.
$$

\begin{definition} \label{21.06.18.1431}

$(i)$ (Parabolic problem)  For given $f \in L_{1,loc}([0,T]\times D)$ and $u_0\in L_{1,loc}(D)$,
we say that $u$ is a  (weak) solution to the problem
\begin{equation} \label{21.06.18.1430}
\begin{cases}
\partial_t u(t,x)=\Delta^{\alpha/2}u(t,x)+f(t,x),\quad &(t,x)\in(0,T)\times D,
\\
u(0,x)=u_0(x),\quad & x\in D,
\\
u(t,x)=0,\quad &(t,x)\in [0,T]\times D^c,
\end{cases}
\end{equation}
if (a) $u=0 \ a.e.$ in $[0,T]\times D^c$, (b) $\langle u(t,\cdot), \phi \rangle_{\bR^d}$ and $ \langle u(t,\cdot), \Delta^{\alpha/2}\phi \rangle_{\bR^d}$ exist   for any $t\leq T$ and test function $\phi\in C_c^\infty(D)$, and (c)  for any  $\phi\in C^{\infty}_c(D)$ the equality
\begin{align} \label{21.09.20.1816}
\langle u(t,\cdot),\phi \rangle_{\bR^d}= \langle u_0,\phi\rangle_D + \int_0^t \langle u(s,\cdot),\Delta^{\alpha/2}\phi \rangle_{\bR^d} ds + \int_0^t \langle f(s,\cdot),\phi \rangle_D ds
\end{align}
holds for all $t\leq T$.  

$(ii)$ (Elliptic problem) Let $\lambda\in[0,\infty)$. For given $f \in L_{1,loc}(D)$,
we say that $u$ is a (weak) solution to
\begin{equation} \label{21.07.13.13.02}
\begin{cases}
\Delta^{\alpha/2}u(x)-\lambda u(x)=f(x),\quad &x\in D,
\\
u(x)=0,\quad &x\in D^c,
\end{cases}
\end{equation}
if  (a) $u=0 \ a.e.$ in $D^c$, (b) $\langle u, \phi\rangle_{\bR^d}$ and $ \langle u, \Delta^{\alpha/2}\phi \rangle_{\bR^d}$ exist   for any  test function $\phi\in C_c^\infty(D)$, and (c)  for any  $\phi\in C^{\infty}_c(D)$  we have
\begin{equation}
\label{eqn para}
\langle u,\Delta^{\alpha/2}\phi \rangle_{\bR^d} - \lambda  \langle u,\phi \rangle_{\bR^d} = \langle f,\phi \rangle_D.
\end{equation}
\end{definition}

It is  clear if $u(t,x)$ is a strong (or point-wise) solution to \eqref{21.06.18.1430} and sufficiently regular,  then $u$ becomes a weak solution in the sense of Definition \ref{21.06.18.1431}.

For an explicit representation of weak solutions, we introduce  some related stochastic processes.  Let $X=(X)_{t\geq 0}$ be a rotationally symmetric $\alpha$-stable $d$-dimensional L\'evy process  defined on a probability space $(\Omega,\cF,\bP)$, that is, $X_t$ is a L\'evy process such that
$$
\bE e^{i \xi \cdot X_t}=e^{-|\xi|^{\alpha}t}, \quad \forall\, \xi\in \bR^d.
$$
Let
$$\tau_{D}=\tau^x_D:=\inf\{t\geq0: x+X_t\not\in D\}
$$
denote  the first exit time of $D$ by $X$.  We add an element, called a cemetery point, $\partial\notin \bR^d$ to $\bR^d$, and define the killed process of $X$ upon $D$ by 
\begin{align*}
X_t^{D}=X_t^{D,x}:=
\begin{cases}
x+X_t\quad &t<\tau^x_{D},
\\
\partial\quad &t\geq \tau^x_{D}.
\end{cases}
\end{align*}
The cemetery point $\partial$ is introduced to define $f(\partial):=0$ for any function $f$  so that $f(X^D_t)=0$ if $t\geq \tau_D^x$.  Let $p^D (t,x,y)$ denote the transition density of $X^{D}$, i.e., for any Borel set $B\subset \bR^d$,
 $$
 \bP(X^{D,x}_t \in B)= \int_B p^D (t,x,y) dy.
 $$
 
Recall $\rho(x)=dist (x, \partial D)$. We denote $L_{p,\theta}(D):=L_p(D, \rho^{\theta-d}dx)$  for any $\theta\in \bR$ and $p>1$. In other words, $L_{p,\theta}(D)$ is the set of functions $u$ such that 
$$
\|u\|_{L_{p,\theta}(D)} := \left(\int_D |u|^p \rho ^{\theta-d} dx\right)^{1/p} <\infty.
$$
For $T<\infty$, we also define the space
$$
\bL_{p,\theta}(D,T):=L_{p}((0,T);L_{p,\theta}(D))
$$
given with the norm
$$
\|u\|_{\bL_{p,\theta}(D,T)}=\left( \int^T_0 \int_D |u|^p \rho^{\theta-d}dx dt \right)^{1/p}.
$$

Here are our uniqueness and existence results of the elliptic and parabolic equations.  The proofs are given in Section \ref{subproof}.

\begin{theorem}[Parabolic case]
\label{21.07.14.13.00}
Let $\alpha\in(0,2)$ and $p\in(1,\infty)$.   Assume  $\theta\in(d-1,d-1+p)$,
$f \in \bL_{p,\theta+\alpha p/2}(D,T)$ and $u_0\in  L_{p,\theta-\alpha p/2 + \alpha}(D)$.

\begin{enumerate}[(i)]

\item The function
\begin{align} \label{rep parabolic}
u(t,x):=\int_D p^D (t,x,y) u_0(y)dy + \int_0^t \int_D p^D (t-s,x,y) f(s,y)dyds
\end{align}
belongs to $\bL_{p,\theta-\alpha p/2}(D,T)\cap \{u=0\,\text{on}\,\,[0,T]\times D^c\}$  and is the unique weak solution to \eqref{21.06.18.1430} in this function  space.

\item For the solution $u$, we have
\begin{equation}
\label{para}
 \|u\|_{\bL_{p,\theta-\alpha p/2}(D,T)} \leq C(\|f\|_{\bL_{p,\theta+\alpha p/2}(D,T)}
 +\|u_0\|_{L_{p,\theta-\alpha p/2 + \alpha}(D)}),
\end{equation} where $C$ is independent of $u$ and $T$.
\end{enumerate}
\end{theorem}

\begin{theorem}[Elliptic case]
\label{21.07.27.12.13}
Let $\alpha\in(0,2)$ and $p\in(1,\infty)$. Assume  $\theta\in(d-1,d-1+p)$ and  $f \in L_{p,\theta+\alpha p/2}(D)$.
\begin{enumerate}[(i)]
    \item Let $\lambda>0$ or $D$ be bounded. Then, the function 
$$
u(x)=u^{(\lambda)}(x):=\int_D \left(\int^{\infty}_0 e^{-\lambda t}p^D(t,x,y) dt \right)f(y)dy
$$
belongs to  $L_{p,\theta-\alpha p/2}(D)\cap \{u=0\,\text{on}\,\,D^c\}$ and is the unique weak solution to \eqref{21.07.13.13.02} in this function space.

\item  Let $\lambda=0$ and $D=\bR^d_+$. Then,  $u^{(1/n)}$ converges weakly in $L_p(\bR^d, \rho^{\theta-d-\alpha p/2}dx)$,  and the weak limit $u$ is the unique solution to equation \eqref{21.07.13.13.02} in the function space 
 $L_{p,\theta-\alpha p/2}(D)\cap \{u=0\,\text{on}\,\,D^c\}$.
 
 \item For the solution $u$, we have
 \begin{equation*}
 \|u\|_{L_{p,\theta-\alpha p/2}(D)}\leq C \|f\|_{L_{p,\theta+\alpha p/2}(D)},
 \end{equation*}
 where  $C$ is independent of $u$ and $\lambda$.
  \end{enumerate}
\end{theorem}

\begin{remark}
By definition of the norm in $\bL_{p,\theta-\alpha p/2}(D,T)$ and \eqref{para},
$$
\|u\|^p_{\bL_{p,\theta-\alpha p/2}(D,T)}=\int^T_0 \int_D |\rho^{-\alpha/2}u|^p \rho^{\theta-d}dxdt<\infty
$$
provided that $-1<\theta-d<-1+p$. This suggests that $u$  vanishes at a certain rate near the boundary of $D$. The detailed behaviors of solutions and their  derivatives   will be handled in the following subsection.
\end{remark}

\begin{remark}
\label{remark sharp}
The range $\theta\in (d-1, d-1+p)$ in Theorems \ref{21.07.14.13.00} and \ref{21.07.27.12.13} is sharp. We demonstrate this with a simple example for the elliptic problem. The parabolic problem can be handled similarly.

 Let $D=B_1(0)$ and $f$ be a (non-zero)   nonnegative function in $C_c^{\infty}(D)$ so that $f\in L_{p,\theta}(D)$ for any $\theta\in\bR$.

1. First, we show $\theta>d-1$ is necessary.  Denote 
$$
G^0_D(x,y):=\int^{\infty}_0 p^D(t,x,y)dt \quad \text{and}\quad  u(x):=\int_{D}G_{D}^{0}(x,y)f(y)dy.
$$
Due to  \cite[Corollary 1.2]{chen2010heat}, if $y\in supp(f)$ and $(r+1)/2<|x|<1$ where $r:=1-dist(supp(f),\partial D)>0$, then $G_{D}^0(x,y)\approx \rho(x)^{\alpha/2}$. Hence, for $(r+1)/2<|x|<1$,
$$
u(x)\approx \rho(x)^{\alpha/2}=(1-|x|)^{\alpha/2},
$$
and consequently 
\begin{align*}
    \| u \|_{L_{p,\theta-\alpha p/2}(D)}^p  \geq C \int_{(r+1)/2}^1(1-s)^{\theta-d}s^{d-1}ds.
\end{align*}
The  right-hand side above is finite only if $\theta-d>-1$. Therefore,  the condition $\theta-d>-1$ is needed to have  $u\in  L_{p,\theta-\alpha p/2}(D)$.

2. Next, we show $\theta<d-1+p$ is also necessary.  Suppose  Theorem \ref{21.07.27.12.13} holds for some  $\theta\geq d-1+p$.  Then, 
\begin{equation*}
\left\|\int_DG_D^0(\cdot,y)g(y)dy\right\|_{L_{p,\theta-\alpha p/2}(D)}\leq C\|g\|_{L_{p,\theta+\alpha p/2}},\quad \forall g\in L_{p,\theta+\alpha p/2}(D).    
\end{equation*}
Since $G_{D}^0(x,y)=G_D^0(y,x)$ (see e.g. \cite[Theorem 2.4]{chung2012brownian}), by H\"older's inequality,
\begin{align*}
    \left|\int_{D}u(x)g(x)dx\right|&=\left| \int_D \left(\int_{D}G_D^0(x,y)g(x)dx\right)f(y)dy\right|\\
    &\leq  \left\|\int_DG_D^0(\cdot,y)g(y)dy\right\|_{L_{p,\theta-\alpha p/2}(D)}\|f\|_{L_{p',\theta'+\alpha p'/2}(D)}
    \\
    &\leq C\|g\|_{L_{p,\theta+\alpha p/2}(D)}\|f\|_{L_{p',\theta'+\alpha p'/2}(D)},
\end{align*}
where $1/p+1/p'=1$ and $\theta/p+\theta'/p'=d$. Since $L_{p',\theta'-\alpha p'/2}(D)$ is the dual space of $L_{p,\theta+\alpha p/2}(D)$ (cf. Lemma \ref{21.06.15.13.47}$(iii)$), this leads to
\begin{equation}
\label{imp}
\|u\|_{L_{p',\theta'-\alpha p'/2}(D)} \leq C\|f\|_{L_{p',\theta'+\alpha p'/2}(D)}<\infty.
\end{equation}
Note that $\theta'\leq d-1$. As shown above, \eqref{imp} is not possible, and therefore we get a contradiction.
\end{remark}

\begin{remark}
Since   $\Delta^{\alpha/2}\phi$ belongs to the dual space of  $L_{p,\theta-\alpha p/2}(D)$ for any $\phi\in C^{\infty}_c(D)$ (cf. Lemma \ref{cor 10.10}) and $C^{\infty}_c(D)$ is dense in the dual space,  we can replace $\langle \cdot, \cdot \rangle_D$ and $\langle \cdot, \cdot \rangle_{\bR^d}$ in \eqref{21.09.20.1816} and \eqref{eqn para} by $(\cdot,\cdot)_D$ for the solutions in Theorems \ref {21.07.14.13.00} and \ref{21.07.27.12.13}.
\end{remark}

\subsection{Regularity of solutions} \label{Regularity of solution}
In this subsection, we present Sobolev regularity of solutions. We also obtain  H\"older estimates of solutions based on a Sobolev embedding theorem.  In particular, we give asymptotic behaviors of solutions and their `arbitrary' order derivatives near the boundary of $D$.

To describe such results,  we first  recall Sobolev and Besov spaces on $\bR^d$.
For $p\in(1,\infty)$ and $\gamma\in\bR$,  the Sobolev space $H_p^\gamma=H_p^{\gamma}(\bR^d)$ is defined as the space of all tempered distributions $f$ on $\bR^d$ satisfying
$$
\|f\|_{H_p^{\gamma}}:=\|(1-\Delta)^{\gamma/2}f\|_{L_p}<\infty,
$$
where
$$
(1-\Delta)^{\gamma/2} f(x) := \cF^{-1} \left[(1+|\cdot|^2)^{\gamma/2}\cF [f] \right](x).
$$
Here, $\cF$ and $\cF^{-1}$ denote the $d$-dimensional Fourier transform and the inverse Fourier transform respectively, i.e.,
$$
\cF[f](\xi):=\int_{\bR^d} e^{-i\xi\cdot x} f(x) dx, \quad  \cF^{-1}[f](x):=\frac{1}{(2\pi)^d}\int_{\bR^d} e^{i\xi\cdot x} f(\xi) d\xi.
$$
As is well known, if  $\gamma\in \bN_+$, then  we have 
$$H^{\gamma}_p=W^{\gamma}_p:=\{f: D^{\beta}_x u\in L_p(\bR^d), |\beta|\leq \gamma\}.
$$
For $T\in(0,\infty)$, define 
$$\bH_p^{\gamma}(T):=L_p((0,T);H_p^{\gamma}),\quad \bL_p(T):=\bH_p^0(T)=L_p((0,T);L_p).
$$
Now we take a function $\Psi$ whose Fourier transform $ \cF[\Psi]$ is infinitely differentiable, supported in an annulus $\{\xi\in\bR^d : \frac{1}{2} \leq |\xi| \leq 2\}$, $\cF[\Psi]\geq0$ and
$$
\sum_{j\in \bZ} \cF[\Psi](2^{-j}\xi)=1, \qquad \forall \xi\neq0.
$$
For a tempered distribution $f$ and $j\in \bZ$, define
$$
\Delta_j f(x):=\cF^{-1}\left[\cF[\Psi](2^{-j}\cdot)\cF [f]\right](x), \qquad 
S_0 f(x):= \sum_{ j=-\infty}^0 \Delta_j f(x).
$$
 The Besov space $B_p^\gamma=B_p^\gamma(\bR^d)$, where $p>1, \gamma\in \bR$,  is defined as the space of all tempered distributions $f$ satisfying
$$
\|f\|_{B_p^\gamma}:=\| S_0 f\|_{L_p} + \left(\sum_{j=1}^\infty 2^{\gamma p j} \| \Delta_j f \|_{L_p}^p \right)^{1/p} < \infty.
$$
It is well known (see e.g. \cite[Remark 2.5.12/2]{triebel2010theory}) that if $\gamma=n+\delta$, where $n\in \bN_+$ and $\delta\in (0,1)$, then
\begin{equation}
\label{21.09.13.14.40.1}
    \|f\|_{B_p^{\gamma}}\approx \|f\|_{H^n_p}
    + \left(\sum_{|\beta|=n}\int_{\bR^d} \int_{\bR^d} 
    \frac{|D^{\beta}_xf(x+y)-D^{\beta}_xf(x)|^p} { |y|^{d+\delta p} }
     dydx \right)^{1/p}.
\end{equation}
Moreover, for any $p>1$, we have 
\begin{equation}
\label{besov}
H^{\gamma_2}_p\subset B^{\gamma_1}_p \quad \text{if}\quad \gamma_1<\gamma_2.
\end{equation}

Next, we introduce weighted Sobolev and Besov spaces on $D\subset\bR^d$. Recall $\rho(x)=dist\,(x,\partial D)$ and $L_{p,\theta}(D):=L_p(D, \rho^{\theta-d}dx)$. For  any $\theta\in \bR$ and $n\in \bN_+$, define 
$$
H^{n}_{p,\theta}(D)=\{u: u, \rho D_x u, \cdots, \rho ^n D^n_x u\in L_{p,\theta}(D)\}.
$$
The norm in this space is defined as
\begin{align} \label{21.10.06.0918}
\|u\|_{H_{p,\theta}^{n}(D)}=\sum_{|\beta|\leq n} \left( \int_{D}|\rho ^{|\beta|}D_x^{\beta}u(x)|^p \rho^{\theta-d}dx\right)^{1/p}.
\end{align}
To generalize this space and define $H^{\gamma}_{p,\theta}(D)$ for any $\gamma\in \bR$, we proceed as follows. 
We choose a sequence of nonnegative functions $\zeta_n \in C^{\infty} (D), n\in \bZ$,  having the following properties:
\begin{align}
    &\label{21.10.03.18.06.1}(i)\,\,supp (\zeta_n) \subset \{x\in D : k_1e^{-n}< \rho(x)<k_2e^{-n}\}, \quad k_2>k_1>0,
    \\
    &\label{21.10.03.18.06.2}(ii)\,\,\sup_{x\in\bR^d}|D^m_x \zeta_n (x)| \leq C(m)e^{mn},\quad \forall m\in\bN_+
    \\
    &\label{21.10.03.18.06.3}(iii)\,\,\sum_{n\in\bZ} \zeta_n(x) > c>0,\quad\forall x\in D.
\end{align}
Such functions can be easily constructed by considering  mollifications of indicator functions of the sets of the type $\{x\in D: k_3e^{-n}< \rho(x) <k_4e^{-n}\}$. If the set $\{x\in D: k_1e^{-n}< \rho(x) <k_2e^{-n}\}$ is empty, we just take $\zeta_n=0$.

Now we define weighted Sobolev spaces $H^{\gamma}_{p,\theta}(D)$ and weighed Besov spaces $B^{\gamma}_{p,\theta}(D)$ for any $\gamma, \theta\in \bR$ and $p>1$.  To understand these spaces, one needs to notice that for any distribution $u$ on $D$, $\zeta_{-n}u$ becomes a distribution on $\bR^d$. Obviously, the action  of $\zeta_{-n}u$ on $C^{\infty}_c(\bR^d)$ is defined as 
\begin{equation}
\label{distribution}
(\zeta_{-n}u, \phi)_{\bR^d}=(u, \zeta_{-n}\phi)_D, \quad \phi \in C^{\infty}_c(\bR^d).
\end{equation}

By $H_{p,\theta}^{\gamma}(D)$ and $B_{p,\theta}^{\gamma}(D)$  we denote the sets of  distributions $u$ on $D$ such that
\begin{equation}
\label{21.10.06.15.13}
\|u\|_{H_{p,\theta}^{\gamma}(D)}^p:=\sum_{n\in\bZ}e^{n\theta}\|\zeta_{-n}(e^n\cdot)u(e^n\cdot)\|_{H_p^{\gamma}}^p<\infty,
\end{equation}
and
\begin{equation*}
\|u\|_{B_{p,\theta}^{\gamma}(D)}^p:=\sum_{n\in\bZ}e^{n\theta}\|\zeta_{-n}(e^n\cdot)u(e^n\cdot)\|_{B_{p}^{\gamma}}^p<\infty,
\end{equation*}
respectively.
  The spaces $H_{p,\theta}^{\gamma}(D)$ and $B_{p,\theta}^{\gamma}(D)$ are independent of choice of $\{\zeta_n\}$ (see e.g. \cite[Proposition 2.2]{lototsky2000sobolev}). More precisely, if $\{\xi_n\in C^{\infty}(D):n\in \bZ\}$ satisfies \eqref{21.10.03.18.06.1} and \eqref{21.10.03.18.06.2}, then
\begin{equation}
\label{equiv norm1}
\sum_{n\in\bZ}e^{n\theta}\|\xi_{-n}(e^n\cdot)u(e^n\cdot)\|_{H_p^{\gamma}}^p\leq C\|u\|_{H_{p,\theta}^{\gamma}(D)}^p,
\end{equation}
and the reverse inequality of \eqref{equiv norm1} also holds if $\{\xi_n\}$ satisfies \eqref{21.10.03.18.06.3}. The similar statements hold in the space $B^{\gamma}_{p,\theta}(D)$ as well.  Furthermore, if $\gamma=n \in \bN_+$,  then  the norms defined in \eqref{21.10.06.0918}  and \eqref{21.10.06.15.13} are equivalent  (cf. \cite[Proposition 2.2]{lototsky2000sobolev}).

Obviously, by \eqref{besov}, we have for any $p>1$ and $\theta\in \bR$,
\begin{equation}
\label{besov weight}
H^{\gamma_2}_{p,\theta}(D) \subset B^{\gamma_1}_{p,\theta}(D) \quad \text{if}\quad \gamma_1<\gamma_2.
\end{equation}
Furthermore, for an equivalent norm in  $B^{\gamma}_{p,\theta}(D)$, we can apply  \eqref{21.09.13.14.40.1} and prove the following: if $\gamma=n+\delta>0$, where $n\in \bN_+$, $\delta\in (0,1)$, and $\theta-d+\gamma p >-1$, then 
\begin{equation}
 \label{eqn rel}
\|u\|_{B^{\gamma}_{p,\theta}(D)} \approx  \|u\|_{H^n_{p,\theta}(D)} + \left( \sum_{|\beta|=n} \int_D\int_{D} \rho_{x,y}^{\theta-d+\gamma p} 
\frac{|D^{\beta}u(x)-D^{\beta}u(y)|^p}{|x-y|^{d+\delta p} }  dydx \right)^{1/p},
\end{equation}
where $\rho_{x,y}=\rho(x) \wedge \rho(y)$. The proof of \eqref{eqn rel} is left to  the reader.  Relation \eqref{eqn rel} will not be used elsewhere in this article.

Next, we choose (cf. \cite{KK2004}) an infinitely differentiable  function $\psi$ in $D$  such that $\psi\approx \rho$ on $D$, and for any $m\in \bN_+$
$$
\sup_D |\rho^m(x) D^{m+1}_x \psi(x)|\leq C(m)<\infty.
$$
 For instance, one can  take $\psi(x):=\sum_{n\in\bZ} e^{-n} \zeta_n(x)$.

Below we  collect some other properties of the spaces $H^{\gamma}_{p,\theta}(D)$ and $B_{p,\theta}^\gamma(D)$.  For $\nu\in \bR$, we write $u\in \psi^{-\nu} H_{p,\theta}^{\gamma}(D)$ (resp. $u\in \psi^{-\nu} B_{p,\theta}^{\gamma}(D)$) if $\psi^{\nu}u \in  H_{p,\theta}^{\gamma}(D)$ (resp. $\psi^{\nu}u \in  B_{p,\theta}^{\gamma}(D)$).

\begin{lemma}\label{21.06.15.13.47}
Let $\gamma,\theta\in\bR$ and $p\in(1,\infty)$.
\begin{enumerate}[(i)]

\item\label{21.06.15.13.47.2}
The space $C^{\infty}_c(D)$ is dense in $H^{\gamma}_{p,\theta}(D)$ and $B^{\gamma}_{p,\theta}(D)$.

\item\label{21.06.15.13.47.5}
For $\delta\in\bR$, $H_{p,\theta}^\gamma (D) = \psi^\delta H_{p,\theta+\delta p}^\gamma(D)$ and $B_{p,\theta}^\gamma (D) = \psi^\delta B_{p,\theta+\delta p}^\gamma(D)$. Moreover,
$$
\|u\|_{H_{p,\theta}^\gamma(D)}\approx \|\psi^{-\delta} u\|_{H_{p,\theta+\delta p}^\gamma(D)}, \quad \|u\|_{B_{p,\theta}^\gamma(D)}\approx \|\psi^{-\delta} u\|_{B_{p,\theta+\delta p}^\gamma(D)}.
$$

\item\label{21.06.15.13.47.9}(Duality)
Let
\begin{equation*} \label{dual}
1/p+1/p'=1, \quad \theta/p+\theta'/p' = d.
\end{equation*}
Then, the dual spaces of $H_{p,\theta}^\gamma(D)$ and $B_{p,\theta}^\gamma(D)$ are $H_{p',\theta'}^{-\gamma}(D)$ and $B_{p',\theta'}^{-\gamma}(D)$, respectively.

\item \label{23.04.14-1}(Sobolev embedding)
Let $\mu\leq\gamma$, $1<p\leq q$ and $\theta\leq\tau$ such that
\begin{equation*} 
  \mu-d/q\leq \gamma-d/p, \quad \tau/q=\theta/p.
\end{equation*}
Then, we have
\begin{equation*}
  \|u\|_{H_{q,\tau}^\mu(D)} \leq C \|u\|_{H_{p,\theta}^\gamma(D)}.
\end{equation*}

\item \label{21.06.15.13.47.8}(Sobolev-H\"older embedding)

Let
$\gamma-\frac{d}{p} \geq n+\delta$ for some $n\in\bN_+$ and $\delta\in(0,1)$. Then, for any $k\leq n$,
\begin{equation*}
    |\psi^{k+\frac{\theta}{p}} D_x^k u|_{C(D)}+[\psi^{n+\frac{\theta}{p}+\delta} D_x^{n}u]_{C^{\delta}(D)}\leq C(d,\gamma,p,\theta)\|u\|_{H_{p,\theta}^{\gamma}(D)}.
\end{equation*}

\end{enumerate}

\end{lemma}
\begin{proof}
The proofs for $B_{p,\theta}^{\gamma}(D)$ are similar to those for $H_{p,\theta}^{\gamma}(D)$, and we only consider the claims for $H_{p,\theta}^{\gamma}(D)$. 
When $D$ is a half space, the claims are proved by Krylov in \cite[Lemma 2.2, Theorem 2.5]{Krysome19}, and those are generalized by Lototsky in \cite{lototsky2000sobolev} for arbitrary domains. Here, we remark that the results in \cite{lototsky2000sobolev} are still valid for bounded $C^{1,1}$ open sets. The lemma is proved.

\end{proof}

Now we define  solution spaces for the parabolic equation.  For $T\in(0,\infty)$, denote
\begin{equation*}
\begin{aligned}
\bH_{p,\theta}^\gamma(D,T) := L_{p}((0,T);H_{p,\theta}^\gamma(D)).
\end{aligned}
\end{equation*}
We write $u\in \frH_{p,\theta}^{\gamma}(D,T)$ if $u\in \psi^{\alpha/2}\bH_{p,\theta}^{\gamma}(D,T)$,   $u(0,\cdot) \in \psi^{\alpha/2-\alpha/p} B_{p,\theta}^{\gamma-\alpha /p}(D)$, and  there exists $f\in \psi^{-\alpha/2}\bH_{p,\theta}^{\gamma-\alpha} (D,T)$ such that for any $\phi\in C_c^\infty(D)$ 
$$
(u(t,\cdot),\phi)_D=(u(0,\cdot),\phi)_D+\int_0^t (f(s,\cdot),\phi)_D ds, \quad \forall \, t\leq T.
$$
In this case, we write $u_t:=\partial_t u:=f$.  The norm in  $ \frH_{p,\theta}^{\gamma}(D,T)$ is defined as 
\begin{align}
\label{u_t}
\|u\|_{\frH_{p,\theta}^{\gamma}(D,T)} :=& \|\psi^{-\alpha/2} u\|_{\bH_{p,\theta}^{\gamma}(D,T)} + \|\psi^{\alpha/2} u_t\|_{\bH_{p,\theta}^{\gamma-\alpha}(D,T)}\\
& + \|\psi^{-\alpha/2+\alpha/p} u(0,\cdot) \|_{B_{p,\theta}^{\gamma-\alpha /p}(D)}. \nonumber
\end{align}

\begin{remark} \label{21.10.22.2123}
$(i)$ The Banach space $\frH_{p,\theta}^{\gamma}(D,T)$ is a modification of the corresponding space defined for $\alpha=2$ (see  e.g. \cite{KK2004} for $C^1$ domains and  \cite{krylov2001some} for a half space). The completeness of this space for  $\alpha\in (0,2)$ can be proved by repeating the argument in \cite[Remark 3.8]{krylov2001some}.

$(ii)$ The same argument in \cite[Remark 5.5]{kry99weighted}  shows that $C^{\infty}_c([0,T]\times D)$ is dense in $\frH^{\gamma}_{p,\theta}(D,T)$.
\end{remark}

The following  two theorems address our Sobolev regularity results.  The proofs are given in Section \ref{subproof}.

\begin{theorem}[Parabolic case]
\label{21.10.18.2056}
Let   $\gamma\in[0,\infty)$, and assume $f \in \psi^{-\alpha/2}\bH_{p,\theta}^{\gamma}(D,T)$ and $u_0\in \psi^{\alpha/2-\alpha/p} B_{p,\theta}^{\gamma+\alpha-\alpha/p}(D)$.  The  unique solution $u$ in Theorem \ref{21.07.14.13.00} belongs to $\frH_{p,\theta}^{\gamma+\alpha}(D,T)$, and for this solution we have
\begin{align} \label{21.06.22.1912}
     \|u\|_{\frH_{p,\theta}^{\gamma+\alpha}(D,T)} \leq C \left( \|\psi^{\alpha/2} f\|_{\bH_{p,\theta}^{\gamma}(D,T)} + \|\psi^{-\alpha/2+\alpha/p} u_0 \|_{B_{p,\theta}^{\gamma+\alpha-\alpha/p}(D)} \right),
\end{align}
where $C$ depends only on $d,p,\alpha,\gamma,\theta$ and $D$.
\end{theorem}

\begin{theorem}[Elliptic case]
\label{21.10.18.2057}
Let  $\gamma,\lambda\in[0,\infty)$ and assume $f \in \psi^{-\alpha/2}H_{p,\theta}^{\gamma}(D)$.  Then, the unique  solution $u$ in Theorem \ref{21.07.27.12.13} belongs to $\psi^{\alpha/2}H_{p,\theta}^{\gamma+\alpha}(D)$, and for this solution we have
\begin{align} \label{21.07.13.13.31}
    \lambda \| \psi^{\alpha/2}u\|_{H_{p,\theta}^{\gamma}(D)}+ \|\psi^{-\alpha/2}u\|_{H_{p,\theta}^{\gamma+\alpha}(D)} \leq C \|\psi^{\alpha/2} f\|_{H_{p,\theta}^{\gamma}(D)},
\end{align}
where $C$ depends only on $d,p,\alpha,\gamma,\theta$ and $D$. In particular, it is independent of $\lambda$.
\end{theorem}

\begin{remark} \label{230315}
$(i)$ Let $\gamma+\alpha \geq n$, where $n\in \bN_+$. Then, \eqref{21.10.06.0918} and \eqref{21.07.13.13.31} certainly yield
$$
\int_D \left(|\rho^{-\alpha/2}u|^p+|\rho^{1-\alpha/2}Du|^p+\cdots+|\rho^{n-\alpha/2}D^nu|^p \right) \rho^{\theta-d}dx < \infty.
$$

$(ii)$ The  parabolic version  of $(i)$ also holds.

$(iii)$ Let $H_{p,loc}^\gamma(D)$ denote the space of all distributions on $D$ such that $u\eta\in H_p^\gamma$ for any $\eta \in C_c^\infty(D)$. Then, due to the definition of $H_{p,\theta}^\gamma(D)$ (see \eqref{21.10.06.15.13}), one can easily find that $H_{p,\theta}^\gamma(D) \subset H_{p,loc}^\gamma(D)$.  Note that if $D$ is bounded, then $ L_p(D)\subset L_{p,d+\alpha p/2}(D)$. Thus, our result directly implies the ones in \cite{zuazua2017,biccari2018local,cozzi2017interior} when $D$ is a bounded $C^{1,1}$ open set, while the latter ones cover a more general class of open sets.
\end{remark}

The following  estimates are consequences of  Lemma \ref{21.06.15.13.47}$(iv)$.

\begin{corollary}
Let $u$ be taken from Theorem \ref{21.10.18.2056} and
\begin{equation*}
  \mu-d/q\leq \gamma+\alpha-d/p, \quad \tau/q=\theta/p.
\end{equation*}
Then, we have
\begin{equation*}
  \|\psi^{-\alpha/2}u\|_{L_p((0,T);H_{q,\tau}^\mu(D))} \leq C \left( \|\psi^{\alpha/2} f\|_{\bH_{p,\theta}^{\gamma}(D,T)} + \|\psi^{-\alpha/2+\alpha/p} u_0 \|_{B_{p,\theta}^{\gamma+\alpha-\alpha/p}(D)} \right).
\end{equation*}
\end{corollary}

\begin{corollary} \label{cor.23.04.14}
  Let $u$ be taken from Theorem \ref{21.10.18.2057} and
\begin{equation*}
  \mu-d/q\leq \gamma+\alpha-d/p, \quad \tau/q=\theta/p.
\end{equation*}
Then, we have
\begin{equation*}
  \|\psi^{-\alpha/2}u\|_{H_{q,\tau}^\mu(D)} \leq C \|\psi^{\alpha/2} f\|_{H_{p,\theta}^{\gamma}(D)}.
\end{equation*}
\end{corollary}

\begin{remark} \label{23.04.14.21}
We compare  Corollary \ref{cor.23.04.14} with the results  in \cite{Abdel Global,leonori2015basic,nowak2020hs}.  Below we assume $D$ is bounded.

$(i)$ Let $q\in(1,\infty)$, $p=\max\{2,\frac{dq}{d+\alpha q/2}\}$ and $f\in L_{p,d+\alpha p/2}(D)$.
Then, 
\begin{equation*}
  \frac{\alpha}{2}-\frac{d}{q}\leq \alpha-\frac{d}{p}.
\end{equation*}
Thus, by Corollary \ref{cor.23.04.14} and Theorem \ref{21.10.18.2057},
\begin{align*}
  \|\psi^{-\alpha/2}u\|_{H_{q,\tau}^{\alpha/2}(D)} \leq C \|\psi^{\alpha/2}f\|_{L_{p,d}(D)},
\end{align*}
where $\tau/q=d/p$, which implies $u\in H_{q,loc}^{\alpha/2}(D)$. This is proved in \cite{nowak2020hs} given that  $f\in L_p(D)$ and $q>2$ (instead of $q>1$).
Since $L_p(D)\subseteq L_{p,d+\alpha p/2}(D)$,  our result extends the one in \cite{nowak2020hs}, although \cite{nowak2020hs} considered more general domains and non-local equations.

$(ii)$ Let $\alpha/2\leq\beta\leq\alpha$ and $q\geq p$ such that
\begin{equation}
\label{23.04.16.13.49}
  \beta-\frac{d}{q}\leq \alpha-\frac{d}{p},\quad q<\frac{dp}{d-1}.
\end{equation}
In this case, $pd/q\in(d-1,d-1+p)$, which allows us to apply Corollaries \ref{cor.23.04.14} and \ref{cor 10.10-1}$(i)$ to get
\begin{align}
  \|\psi^{\beta-\alpha/2}\Delta^{\beta/2}u\|_{L_{q}(D)}&=\|\psi^{\beta-\alpha/2}\Delta^{\beta/2}u\|_{L_{q,d}(D)} \nonumber\\
  &\leq C \|\psi^{-\alpha/2}u\|_{H_{q,d}^\beta(D)} \nonumber\\
  &\leq  C \|\psi^{\alpha/2}f\|_{L_{p,pd/q}(D)}. \label{23.04.20.01.01}
\end{align}
According to \cite[Theorem 1.4]{Abdel Global},  if $\alpha/2\leq\beta< (1\wedge\alpha)$ and
\begin{equation}
\label{23.04.16.13.50}
  \beta-\frac{d}{q}< \alpha-\frac{d}{p}<\beta,
\end{equation}
then 
\begin{equation*}
\|\psi^{\beta-\alpha/2}\Delta^{\beta/2}u\|_{L_q(D)} \leq C\|f\|_{L_p(D)}.
\end{equation*}
Also, by \cite[Theorem 24]{leonori2015basic}, if 
\begin{align}
\label{23.04.20.00.44}
    \frac{\alpha}{2}-\frac{d}{q}= \alpha-\frac{d}{p},\quad 1<p<\frac{2d}{d+\alpha},
\end{align}
then it holds that 
\begin{align*}
  \|\Delta^{\alpha/4}u\|_{L_q(D)} \leq C \|f\|_{L_p(D)}.
\end{align*}
One can note that \eqref{23.04.16.13.49}, \eqref{23.04.16.13.50}, and \eqref{23.04.20.00.44} are distinct conditions. Note that 
given that $\alpha/2\leq\beta$,  we have $L_p(D) \subset L_{p,d+p(\alpha /2+d/q)}(D)$ for $p,q\in(1,\infty)$ satisfying \eqref{23.04.16.13.49}. Consequently, \eqref{23.04.20.01.01} allows a broader class of data $f$.
\end{remark}

For H\"older regularity of the solution to the parabolic equation, we use the following parabolic embedding.

\begin{prop} \label{21.10.06.0929}
Let  $\alpha\in(0,2)$, $p\in(1,\infty)$, and  $\gamma, \theta\in \bR$. Then, for any  $1/p<\nu\leq1$,
\begin{equation}
 \label{21.08.24.11.17}
 \Big|\psi^{\alpha(\nu-1/2)}\left(u-u(0,\cdot) \right) \Big|_{C^{\nu-1/p}([0,T];H_{p,\theta}^{\gamma+\alpha-\nu \alpha}(D))}\leq C\|u\|_{\frH_{p,\theta}^{\gamma+\alpha}(D,T)},
\end{equation}
where $C$ depends only on $d$, $\nu$, $p$, $\theta$, $\alpha$ and $T$.
\end{prop}

\begin{proof}
We repeat the argument in \cite{krylov2001some} which treats the case $\alpha=2$.  Considering $u-u_0$ in place of $u$, we may assume $u_0=0$. Let $u_t=f$.
By \eqref{21.10.06.15.13} and Lemma \ref{21.06.15.13.47}($\ref{21.06.15.13.47.5}$),
\begin{align}
\nonumber
    &\left|\psi^{\alpha(\nu-1/2)}u\right|^p_{C^{\nu-1/p}([0,T];H^{\gamma+\alpha-\nu\alpha}_{p,\theta}(D))}
\\
&\leq C\sum_{n\in\bZ} e^{n(\theta+p\alpha(\nu-1/2))}
|u(\cdot, e^n\cdot)\zeta_{-n}(e^n\cdot)|^p_{C^{\nu-1/p}([0,T];H^{\gamma+\alpha-\nu\alpha}_{p})}.  \label{10.7.1}
\end{align}
Denote $v_n(t,x)=u(t,e^nx)\zeta_{-n}(e^nx)$. Then, $\partial_t v_n(t,x)=f(t,e^nx)\zeta_{-n}(e^nx)$. Thus,
by Lemma \ref{21.10.06.15.11} with $a=e^{- np\alpha/2}$,
\begin{align*}
 &e^{np\alpha(\nu-1/2)}|u(\cdot,e^n\cdot)\zeta_{-n}(e^n\cdot)|^p_{C^{\nu-1/p}([0,T];H^{\gamma+\alpha-\nu\alpha}_{p})}
 \\
 &\leq Ce^{-np\alpha/2}\|u(\cdot, e^n\cdot)\zeta_{-n}(e^n\cdot)\|^p_{\bH^{\gamma+\alpha}_{p}(T)}  
+Ce^{np\alpha/2}\|f(\cdot, e^n\cdot)\zeta_{-n}(e^n\cdot)\|^p_{\bH^{\gamma}_{p}(T)}. 
\end{align*}
Coming back to \eqref{10.7.1} and using  \eqref{21.10.06.15.13}, 
\begin{align*}
&\left|\psi^{\alpha(\nu-1/2)}u\right|^p_{C^{\nu-1/p}([0,T];H^{\gamma+\alpha-\nu\alpha}_{p,\theta}(D))}
\\
&\leq C\|\psi^{-\alpha/2}u\|^p_{\bH^{\gamma+\alpha}_{p,\theta}(D,T)}
+C\|\psi^{\alpha/2}f\|^p_{\bH^{\gamma}_{p,\theta}(D,T)}.
\end{align*}
This and Lemma \ref{21.06.15.13.47}$(ii)$ prove \eqref{21.08.24.11.17}. 
\end{proof}

Proposition \ref{21.10.06.0929} and Lemma \ref{21.06.15.13.47}($\ref{21.06.15.13.47.8}$) yield the following results.  
\begin{corollary}(H\"older regularity for parabolic equation)
 \label{21.09.22.1621}
Let  $u$  be taken from Theorem \ref{21.10.18.2056}, $1/p<\nu\leq1$,  and 
    $$
    \gamma+\alpha-\nu\alpha-\frac{d}{p} \geq n+\delta, \quad n\in \bN_+, \, \delta\in (0,1).
    $$
   Then, 
    \begin{align*}
        &\sum_{k=0}^n|\psi^{k+\frac{\theta}{p}+\alpha\left(\nu-\frac{1}{2}\right)}D^k_x(u-u(0,\cdot))|_{C^{\nu-1/p}([0,T];C(D))}
        \\
        &+\sup_{t,s\in[0,T]}\frac{[\psi^{n+\delta+\frac{\theta}{p}+\alpha\left(\nu-\frac{1}{2}\right)}D^n_x(u(t,\cdot)-u(s,\cdot))]_{C^{\delta}(D)}}{|t-s|^{\nu-1/p}}\leq C\|u\|_{\frH_{p,\theta}^{\gamma+\alpha}(D,T)}.
    \end{align*}

\end{corollary}

\begin{corollary}(H\"older regularity for elliptic equation)
\label{cor ellipitic}
Let $u$ be taken from Theorem \ref{21.10.18.2057}   and 
    $$
    \gamma+\alpha-\frac{d}{p}\geq n+\delta, \quad n\in \bN_+, \, \delta\in (0,1).
    $$
Then,
    \begin{align*}
        \sum_{k=0}^n|\psi^{k+\frac{\theta}{p}-\frac{\alpha}{2}}D^k_xu|_{C(D)}+[\psi^{n+\delta+\frac{\theta}{p}-\frac{\alpha}{2}}D^n_xu]_{C^{\delta}(D)}\leq C\|\psi^{-\alpha/2}u\|_{H_{p,\theta}^{\gamma+\alpha}(D)}.
    \end{align*}
\end{corollary}

\begin{remark}  \label{rem 10.27.1} Corollaries \ref{21.09.22.1621} and \ref{cor ellipitic} give  various H\"older estimates of solutions and their arbitrary order derivatives.
Below we elaborate some special cases.  We only consider $\gamma=0,1,2,\cdots$ and $\theta=d$. Note  $L_{p,d}(D)=L_p(D)$.

$(i)$ Parabolic H\"older estimates  when $\gamma=0$. Let $u_0=0$ for simplicity, and assume  $\psi^{\alpha/2}f \in \cap_{p>d/\alpha} \bL_{p,d}(D,T)$.  Obviously this holds  e.g. if $D$ is bounded and $\psi^{\alpha/2}f \in L_{\infty}([0,T]\times D)$.  Taking $\nu \uparrow 1$ and  $p \uparrow \infty$,  from Corollary \ref{21.09.22.1621}  we get
\begin{align*}
    \sup_{x\in D}|\psi^{\alpha/2-\delta}(x)u(\cdot,x)|_{C^{1-\varepsilon}([0,T])}<\infty
\end{align*}
for any small $\delta,\varepsilon>0$. This gives maximal regularity with respect to time variable.  Now, we take $p$ sufficiently large and $\nu$ sufficiently close to $1/p$ to get 
\begin{align*}
    \sup_{x\in D}|\psi^{-\alpha/2+\delta'}(x)u(\cdot,x)|_{C^{\varepsilon'}([0,T])}+\sup_{t\in [0,T]}|\psi^{\alpha/2-\delta'}u(t,\cdot)|_{C^{\alpha-\varepsilon'}(D)}<\infty
\end{align*}
for any small $\delta',\varepsilon'>0$. The second term above gives the maximal interior regularity with respect to space variable, and the first one gives a decay rate near the boundary of $D$. In particular,
$$
\sup_{t\in[0,T]} |u(t,x)|\leq C(\delta') \psi^{\alpha/2-\delta'}(x), \quad \forall \delta'>0.
$$

$(ii)$ Elliptic H\"older estimates when $\gamma=0$. Let $\psi^{\alpha/2}f \in \cap_{p>d/{\alpha}}L_{p,d}(D)$.  Taking  $p$ sufficiently large,  from Corollary \ref{cor ellipitic} we get
\begin{equation}
\label{holder 1}
|\psi^{\alpha/2}u|_{C^{\alpha-\varepsilon}(D)} +|u|_{C^{\alpha/2-\varepsilon}(D)}+|\psi^{-\alpha/2+ \delta}u|_{C^{\varepsilon}(D)}<\infty
\end{equation}
for any small $\delta,\varepsilon >0$.  In \cite{ros2014dirichlet}, it is proved that if $\lambda=0$ and $f\in L_\infty (D)$, then  
\begin{equation}
\label{holder}
|u|_{C^{\alpha/2}(D)}+|\psi^{-\alpha/2}u|_{C^{\beta}(D)}<\infty
\end{equation}
for some $\beta>0$. Thus, there is a slight gap between \eqref{holder 1} and \eqref{holder}. However, our result holds even when $f$ blows up near the boundary since we  assume (at most) $\psi^{\alpha/2}f$ is bounded. 

$(iii)$ Higher order estimates. Let $\gamma=n\in \bN$.  Then, the same arguments above show that all the claims in ($i$)-($ii$)  also hold  for  $\psi D_xu, \psi^2 D^2_xu, \cdots, \psi^n D^n_x u$. That is, the estimates hold if one replaces $u$ by any of these functions.   In particular, if $\psi^{\alpha/2}f, \psi^{\alpha/2+1}D_xf \in  \cap_{p>d/{\alpha}}L_{p,d}(D)$, then, together with \eqref{holder 1},  we  also have 
$$
|\psi^{1+\alpha/2}D_xu|_{C^{\alpha-\varepsilon}(D)} +|\psi D_x u|_{C^{\alpha/2-\varepsilon}(D)}+|\psi^{1-\alpha/2+ \delta}D_xu|_{C^{\varepsilon}(D)}<\infty
$$
for any small $\delta,\varepsilon >0$.
\end{remark}

\mysection{The zero-th order derivative  estimates} \label{zero order}

In this section, we estimate the zero-th order derivative of the solutions to the  parabolic equation 
\begin{equation} \label{21.06.15.1123}
\begin{cases}
\partial_t u(t,x)=\Delta^{\alpha/2}u(t,x)+f(t,x),\quad &(t,x)\in(0,T)\times D,\\
u(0,x)=u_0(x),\quad & x\in D,\\
u(t,x)=0,\quad &(t,x)\in[0,T]\times D^c.
\end{cases}
\end{equation}
as well as to the elliptic equation 
\begin{equation} \label{21.07.26.16.22}
\begin{cases}
\Delta^{\alpha/2}u(x)-\lambda u(x)=f(x),\quad &x\in D,\\
u(x)=0,\quad & x\in D^c.
\end{cases}
\end{equation}

\subsection{Weak solutions for smooth data} \label{21.10.03.1127}
Recall that $X=(X)_{t\geq 0}$ is a rotationally symmetric $\alpha$-stable $d$-dimensional L\'evy process.
Let $p(t,x)=p_d(t,x)$ denote the transition density function of $X$.  Then, it is well known (e.g. \cite[(3.6)]{kim2021lq}) that
\begin{align*}
p_d(t,x)&=(2\pi)^{-d}\int_{\bR^d}e^{ix\cdot\xi}e^{-t|\xi|^{\alpha}}d\xi \nonumber
\\
&\approx t^{-\frac{d}{\alpha}}\wedge \frac{t}{|x|^{d+\alpha}} \approx \frac{t}{(t^{1/\alpha}+|x|)^{d+\alpha}}, \quad \forall (t,x)\in(0,\infty)\times\bR^d.
\end{align*}
The equality above also implies that $p_d(t,\cdot)$ is a radial function and
\begin{equation*}
p_d(t,x)=t^{-\frac{d}{\alpha}}p_d(1,t^{-\frac{1}{\alpha}}x).
\end{equation*}

Denote
$$
d_x=d_{D,x}:=\begin{cases}\rho(x)\, &: \, x\in D, \\
0\,&:\, x\not\in D.
\end{cases}
$$
The following lemma gives an upper bound of $p^D (t,x,y)$. 

\begin{lemma} \label{21.06.15.1700}
For any $x,y\in \bR^d$,
\begin{align*} 
    p^D (t,x,y) 
 \leq
    \begin{cases}
     C\left(1\wedge \frac{d_x^{\alpha/2}}{\sqrt{t}}\right)\left(1\wedge \frac{d_y^{\alpha/2}}{\sqrt{t}}\right)p(t,x-y)\,&\text{if $D$ is a half space},
     \\
     Ce^{-c t}\left(1\wedge \frac{d_x^{\alpha/2}}{\sqrt{t}}\right)\left(1\wedge \frac{d_y^{\alpha/2}}{\sqrt{t}}\right)p(t,x-y)\, &\text{if $D$ is   bounded}.
    \end{cases}
\end{align*}
Here, $C,c>0$ depend only on $d,\alpha$ and $D$. \end{lemma}

\begin{proof}
See \cite[Theorem 5.8]{bogdan2014dirichlet} for the case $D=\bR^d_+$. Let $D$ be  bounded.  Then, by \cite[Theorem 4.5]{bogdan2014dirichlet}, there exist $C,c,r>0$, depending only on $\alpha,d$ and $D$, such that for any $x,y\in D$
\begin{align*}
p^{D}(t,x,y)\leq C e^{-2c t} \left(\frac{d_x^{\alpha/2}}{\sqrt{t}\wedge r^{\alpha/2}}\wedge 1\right)\left(\frac{d_y^{\alpha/2}}{\sqrt{t} \wedge r^{\alpha/2}}\wedge 1\right) p(t\wedge r^{\alpha},x-y).
\end{align*}
This actually implies the claim of the lemma. Indeed, the case  $t<r^{\alpha}$ is obvious, and if $t>r^{\alpha}$ then
$$
p(r^{\alpha},x-y)=r^{-d}p(1,r^{-1}(x-y))\leq r^{-d}p(1,t^{-\frac{1}{\alpha}}(x-y))=r^{-d}t^{d/\alpha}p(t,x-y).
$$
This certainly proves the claim.
The lemma is proved.
\end{proof}

 For  $x\in \bR^d$, we use $\bE_x$ and $\bP_x$ to denote the expectation and distribution of $x+X$. For instance, $\bP_x(X_t\in A):=\bP(x+X_t\in A)$.
Recall that  $f(\partial):=0$ for any function $f$, where $\partial$ is the cemetery point.

Now, we introduce the probabilistic representation of equation \eqref{21.06.15.1123} for smooth data.

\begin{lemma} \label{21.05.12.1839}
\begin{enumerate}[(i)]
\item
Suppose  $u_0\in C_c^\infty( D)$ and $f\in C_c^\infty((0,T)\times D)$. Then,
\begin{align*}
u(t,x)&:=\bE_x[u_0(X_t^{ D})]+\int_0^{t}\bE_x[f(s,X_{t-s}^{D})]ds
\\
&=\int_D p^D (t,x,y) u_0(y)dy + \int_0^t \int_D p^D (t-s,x,y) f(s,y)dyds
\end{align*}
is a  weak solution to \eqref{21.06.15.1123} in the sense of Definition  \ref{21.06.18.1431}(i).

\item Let $u\in C_c^\infty([0,T]\times D)$. Then,
\begin{align} \label{21.05.16.1956}
u(t,x)=\bE_x[u(0,X_t^{ D})]+\int_0^{t}\bE_x[f(s,X_{t-s}^{ D})]ds,
\end{align}
where $f:=\partial_t u - \Delta^{\alpha/2}u$.

\end{enumerate}
\end{lemma}

\begin{proof}

$(i)$ If $u_0=0$, then it follows from  \cite[Lemma 8.4]{zhang2018dirichlet}. The general case is handled similarly.  

$(ii)$ This follows from  \cite[Theorem 5.5]{zhang2018dirichlet}.  We  remark that  \cite[Theorem 5.5]{zhang2018dirichlet} is proved only on  bounded open sets, but the result holds   even on a half space.  Indeed, let  $D_n\subset D=\bR^d_+$ be a sequence of bounded $C^{1,1}$ open sets such that $D_n\uparrow D$ and $supp(u(t,\cdot)) \subset D_n$ for all $t\in[0,T]$. Since $D_n$ is bounded, by  \cite[Theorem 5.5]{zhang2018dirichlet},
\begin{align*}
u(t,x)&=\bE_x[u(0,X_t^{D_n})] +  \bE_x\left[ \int_0^{t\wedge \tau_{D_n}}f(t-s,X_{s})ds\right],
\end{align*}
where $\tau_{D_n}$ is the first exit time of  $D_n$ by $X$, and  $X^{D_n}$ is the killed process of $X$ upon $D_n$ (see Section \ref{Main results}). By following the proof of \cite[Lemma 5.4]{zhang2018dirichlet}, we have $\tau_{D_n}\uparrow \tau_D$. This, since both $u(0)$ and $f$ are bounded,  certainly  yields \eqref{21.05.16.1956}.
The lemma is proved.
\end{proof}

Let $\{T_t\}_{t\geq0}$ and $\{T_t^D\}_{t\geq0}$ be the transition semigroups of $X$ and $X^D$ defined by
\begin{equation*}
    T_tf(x):=\bE_x[f(X_t)], \qquad T_t^D f(x) := \bE_x[f(X_t^D)],
\end{equation*}
respectively. 
    It is known (see e.g. \cite[Example 1.3]{bottcher2013levy} and page 68 of \cite{chung1986doubly}) that $\{T_t\}_{t\geq0}$ and $\{T_t^D\}_{t\geq0}$ are Feller semigroups. For instance, $\{T_t^D\}_{t\geq0}$ is a family of linear operators on $L_\infty(D)$ such that

$(i)$ for any $f\in L_\infty(D)$,
$$
T^D_tT^D_s f=T^D_{t+s}f,
$$

$(ii)$ for any $f\in C_0(D)$, $T^D_t f \in C_0(D)$ and
$$
\lim_{t\to0} \|T^D_t f -f \|_{L_{\infty}(D)}=0.
$$
We also define infinitesimal generators $A$ and $A_D$ by
$$
Af(x):= \lim_{t\downarrow0} \frac{T_t f(x)-f(x)}{t}, \qquad  A_Df(x):= \lim_{t\downarrow0} \frac{T^D_t f(x)-f(x)}{t}
$$
provided that  the limits exist. It is well known (e.g. \cite[Theorem 2.3]{baeumer2018space}) if $f\in C_0(D)$   and one of $Af(x)$ and $A_Df(x)$ exists, then the other also exists and   $Af(x)=A_Df(x)$. Moreover if $f\in C^2_b(\bR^d)$, then $Af(x)=\Delta^{\alpha/2}f(x)$  (e.g. \cite[Lemma 2.6]{baeumer2018space}).

\begin{lemma}
\label{21.07.16.13.21}
  Assume $u\in C_0(D)$ and $Au(x)$ exists for all $x\in D$. If $u$ satisfies
\begin{align} \label{21.06.29.1426}
A u-\lambda u =0 \text{ in } D, \quad \lambda>0,
\end{align}
then, $u\equiv0$.
\end{lemma}

\begin{proof}
Assume $\underset{x\in D}{\sup} u>0$. Since $u\in C_0(D)$, there exists $x_0\in D$ such that
$$
u(x_0)=\sup_{x\in D} u(x).
$$
By the definition of the infinitesimal generator,
$$
A u(x_0)= \lim_{t\downarrow 0} \frac{\bE_{x_0}[u(X_t)]-u(x_0)}{t}\leq 0.
$$
Hence, \eqref{21.06.29.1426} yields  a contradiction. Using the similar argument for $-u$, we conclude that $u\equiv0$. The lemma is proved.
\end{proof}

For $\lambda\geq0$, we define the Green function 
$$
G_D^\lambda(x,y):=\int_0^\infty e^{-\lambda t} p^D(t,x,y)dt.
$$
By  Lemma \ref{21.06.15.1700},    $G_D^\lambda(x,y)$ is well defined if $x\neq y$.

\begin{lemma} \label{21.06.29.1343}
Let $D$ be a half space (resp. a bounded $C^{1,1}$ open set) and $\lambda>0 $ (resp. $\lambda\geq0$). For $f\in C(D)$, define
\begin{equation}
  \label{eqn 10.8.1}
v(x):=\int_{D} G_D^\lambda (x,y)f(y)dy.
\end{equation}
\begin{enumerate}[(i)]
    \item 
$v\in C_0(D)$, $Av(x)$ exists for all $x\in D$, and $v$ is a strong(point-wise) solution to
\begin{equation}  \label{21.09.26.2334}
\begin{cases}
A v(x) -\lambda v(x) = f(x),\quad & x\in D,
\\
v(x)=0,\quad &x \in D^c.
\end{cases}
\end{equation}

\item $v$ is a  weak solution to \eqref{21.07.26.16.22} in the sense of Definition  \ref{21.06.18.1431}(ii).

\item Let $u\in C_c^\infty(D)$ and $g:=\Delta^{\alpha/2}u-\lambda u$. Then,
    \begin{align} \label{21.09.27.0117}
    u(x)=\int_DG_D^{\lambda}(x,y)g(y)dy.
    \end{align}
\end{enumerate}
\end{lemma}
\begin{proof}
$(i)$ The claim follows from \cite[Lemma 3.6]{kim2019boundary} if $D$ is bounded and $\lambda=0$. We repeat its proof for the case $\lambda>0$. 
    First, we show $v\in C_0(D)$.
 By Lemma \ref{21.06.15.1700},
\begin{align*}
    \int_0^\infty  \int_{D} e^{-\lambda t}p^D(t,x,y)|f(y)| dydt    &\leq C\|f\|_{L_\infty(D)} \int_0^\infty  \int_{\bR^d} e^{-\lambda t} p(t,x-y)  dydt
    \\
    &= C\|f\|_{L_\infty(D)} \int_0^\infty e^{-\lambda t}  dt <\infty.
\end{align*}
Thus, by Fubini's theorem,
$$
v(x)=\int_0^\infty e^{-\lambda t} T_t^D f(x) dt.
$$
Since $T_t^D f \in C_0(D)$ and 
$\|T_t^D f\|_{L_{\infty}(D)} \leq  \|f\|_{L_{\infty}(D)}$, 
    the dominated convergence theorem easily yields $v\in C_0(D)$.

Since $\{T_t^D\}_{t\geq0}$ is a Feller semigroup, for any $x\in D$, 
\begin{align*}
    A_D v(x)&:=\lim_{t\downarrow 0} \frac{T_t^D v(x)-v(x)}{t} \quad \text{(provided that the limit exists)}
    \\
    &=\lim_{t\downarrow 0} \frac{1}{t} \left(T_t^D \int_0^\infty e^{-\lambda s} T_s^D f(x) ds - \int_0^\infty e^{-\lambda s} T_s^D f(x) ds\right)
    \\
    &=\lim_{t\downarrow 0} \frac{1}{t} \left( \int_0^\infty e^{-\lambda s} T_{t+s}^D f(x) ds - \int_0^\infty e^{-\lambda s} T_s^D f(x) ds\right)
    \\
    &=\lim_{t\downarrow 0} \frac{1}{t} \left( e^{\lambda t} \int_t^\infty e^{-\lambda s} T_{s}^D f(x) ds - \int_0^\infty e^{-\lambda s} T_s^D f(x) ds\right)
    \\
    &=\lim_{t\downarrow 0} \frac{e^{\lambda t}-1}{t} \int_t^\infty e^{-\lambda s} T_{s}^D f(x) ds + \lim_{t\to0} \frac{1}{t} \left( \int_0^t e^{-\lambda s} T_{s}^D f(x) ds\right)
    \\
    &= \lambda v(x) + f(x).
\end{align*}
Since the limits  exist, we conclude that  $A_Dv$ exists, $A_Dv=Av$, and $v$ satisfies \eqref{21.09.26.2334}. 

$(ii)$ Let $\varphi\in C_c^\infty(D)$. Since $\|T_t^D \varphi\|_{L_{\infty}(D)} \leq  \|\varphi\|_{L_{\infty}(D)}$,
$$
\lim_{t\to\infty} e^{-\lambda t}T_t^D \varphi =0.
$$
Since $\varphi\in C_c^\infty(D)$, we can use the relation $\partial_t T_t^D \varphi=T_t^D \Delta^{\alpha/2} \varphi$ (see \cite[Lemma 8.4]{zhang2018dirichlet}) to get
\begin{align*}
    ( v,\Delta^{\alpha/2}\varphi)_{\bR^d} &= \int_0^\infty ( e^{-\lambda t} T_t^D f,\Delta^{\alpha/2}\varphi)_D dt
    \\
    &=\int_0^\infty ( f,e^{-\lambda t} T_t^D \Delta^{\alpha/2}\varphi)_D dt
    \\
    &=\int_0^\infty ( f,e^{-\lambda t}\partial_t T_t^D \varphi)_D dt
    \\
    &=-\lim_{t\to\infty} ( f,e^{-\lambda t}T_t^D \varphi)_D + ( f, \varphi)_D + \int_0^\infty ( f, \lambda e^{-\lambda t}T_t^D \varphi)_D dt
    \\
    &=(f, \varphi)_D + \lambda ( v,\varphi)_D.
\end{align*}

$(iii)$
Note that $f:=\Delta^{\alpha/2}u-\lambda u\in C(D)$.  Assume $\lambda>0$ for the moment. Take $v(x)$ from \eqref{eqn 10.8.1}.   Then, since $u\in C^2_b(\bR^d)$,  we have $Au=\Delta^{\alpha/2}u$ (e.g. \cite[Lemma 2.6]{baeumer2018space}), and therefore both $u$ and $v$ satisfy the equation
$Aw(x)-\lambda w(x)=f(x)$ for each $x\in D$. 
We conclude $u=v$ due to Lemma \ref{21.07.16.13.21}. If $\lambda=0$ and $D$ is a bounded $C^{1,1}$ open set, then the uniqueness result in \cite[Theorem 3.10]{kim2019boundary} easily yields \eqref{21.09.27.0117}.
The lemma is proved.
\end{proof}

\subsection{Estimates of zero-th order of solutions}

Denote
$$
  \cT_{D}^0u_0(t,x):=\int_D p^D(t,x,y)u_0(y)dy,
  $$
  $$
   \cT_Df(t,x)
    :=\int_0^t \int_D p^D(t-s,x,y)f(s,y) dyds,
$$
$$
\cG_D^{\lambda}f(x):=\int_{D}G^{\lambda}_D(x,y)f(y)dy.
$$
In this subsection,  we prove the operators
\begin{align*}
    &\cT_D^0:\psi^{\alpha/2-\alpha/p}L_{p,\theta}(D)\to\psi^{\alpha/2}\bL_{p,\theta}(D,T),\\
    &\cT_D:\psi^{-\alpha/2}\bL_{p,\theta}(D,T)\to\psi^{\alpha/2}\bL_{p,\theta}(D,T),\\
    &\cG_D^{\lambda}:\psi^{-\alpha/2}L_{p,\theta}(D)\to\psi^{\alpha/2}L_{p,\theta}(D)
\end{align*}
are bounded.  Our proofs highly depend on  the following lemma, which is proved  in Lemma \ref{re2}.

\begin{lemma}
\label{21.05.17.15.05}
Let $\alpha\in(0,2)$, $\gamma_0$, $\gamma_1\in\bR$.
Suppose that
\begin{equation*}
-\frac{2}{\alpha}<\gamma_0,\quad -2<\gamma_1-\gamma_0\leq2+\frac{2}{\alpha}.
\end{equation*}
Then,    for any $(t,x)\in (0,\infty)\times\bR^d$,
    \begin{align*}
        \int_{D}p(t,x-y)\frac{d_y^{\gamma_0\alpha/2}}{(\sqrt{t}+d_y^{\alpha/2})^{\gamma_1}}dy\leq C (\sqrt{t}+d_x^{\alpha/2})^{\gamma_0-\gamma_1},
    \end{align*}
    where $C=C(d,\alpha,\gamma_0,\gamma_1, D)$.
     \end{lemma}

We first consider the operator $\cT_D$.

\begin{lemma} \label{21.06.18.1403}
Let $\alpha\in(0,2)$ and $p\in(1,\infty)$. Suppose that
$$
d-1<\theta<d-1+p.
$$
Then, there exists $C=C(d,\alpha,\theta,p,D)$ such that for any $f\in \psi^{-\alpha/2}\bL_{p,\theta}(D,T)$,
\begin{align*}
\|\psi^{-\alpha/2}\cT_Df\|_{\bL_{p,\theta}(D,T)}\leq C\|\psi^{\alpha/2}f\|_{\bL_{p,\theta}(D,T)}.
\end{align*}
\end{lemma}
\begin{proof}
By Lemma \ref{21.06.15.13.47}$(\ref{21.06.15.13.47.5})$ and \eqref{21.10.06.0918}, it suffices to show
\begin{align} \label{21.08.10.1111}
\int_0^T \int_D d_x^{\mu-\alpha p/2}|\cT_Df (t,x)|^p dxdt \leq C\int_0^T \int_D d_x^{\mu+\alpha p/2}|f(t,x)|^p dxdt,
\end{align}
where $\mu:=\theta-d$. For $p'=p/(p-1)$, since $\mu\in (-1,p-1)$, we can take $\beta_0$ satisfying
\begin{align}
\label{21.05.24.23.44}
\frac{2\mu}{p \alpha} + 1 - \frac{4}{p} < \beta_0 < \frac{2\mu}{p \alpha} + 1 + \frac{2}{p\alpha}
\end{align}
and
\begin{align}
\label{21.05.24.23.45}
-\frac{2(p-1)}{p}=-\frac{2}{p'}< \beta_0 <\left(2+\frac{2}{\alpha}\right)\frac{1}{p'}=\left(2+\frac{2}{\alpha}\right)\frac{p-1}{p}.
\end{align}
Since  $1 - \frac{2}{p}<  \frac{2\mu}{p \alpha}+ 1+\frac{2}{p\alpha}  - \frac{2}{p}$ and  $\frac{2\mu}{p \alpha} +1 <  \frac{2(p-1)}{p\alpha}+1$, we can take constants $\beta_1$ and $\beta_2$ such that 
\begin{align} \label{21.08.10.1109}
   1 - \frac{2}{p}< \beta_0 - \beta_1 < \frac{2\mu}{p \alpha}+ 1+\frac{2}{p\alpha}  - \frac{2}{p}  \end{align}
and
\begin{align} \label{21.08.10.1110}
    \frac{2\mu}{p \alpha} +1 < \beta_0+\beta_2 < \frac{2(p-1)}{p\alpha}+1.
\end{align}
Let $R_{t,x} := \frac{d_x^{\alpha/2}}{\sqrt{t}+d_x^{\alpha/2}}$. By Lemma \ref{21.06.15.1700} and H\"older's inequality,
\begin{align} \label{21.06.18.1054}
|\cT_{D} f(t,x)|\leq& C\left(\int_0^t \int_{D} p(t-s,x-y) d_y^{-\alpha \beta_0 p'/2} R_{t-s,x}^{(1-\beta_1)p'} R_{t-s,y}^{(1-\beta_2)p'} dyds \right)^{1/{p'}} \nonumber
\\
&\times \left(\int_0^t \int_{D}  p(t-s,x-y)d_y^{\alpha \beta_0 p/2}  R_{t-s,x}^{\beta_1 p} R_{t-s,y}^{\beta_2 p} |f(s,y)|^p dyds\right)^{1/p} \nonumber
\\
=:&C\times I(t,x)\times II(t,x).
\end{align}
By Lemma \ref{21.05.17.15.05} with $\gamma_0=(1-\beta_2)p'-\beta_0 p'$ and $ \gamma_1=(1-\beta_2)p'$, we have
\begin{align*}
    \int_{D} p(t-s,x-y) d_y^{-\alpha \beta_0 p'/2}R_{t-s,y}^{(1-\beta_2)p'} dy &\leq C (\sqrt{t-s} + d_x^{\alpha/2} )^{-\beta_0 p'}.
\end{align*}
Using this inequality and changing variables,
\begin{align} \label{21.06.18.1053}
    I(t,x)^{p'}&\leq C d_x^{\alpha(1-\beta_1)p'/2} \int_0^t (\sqrt{t-s} + d_x^{\alpha/2} )^{-\beta_0 p'-(1-\beta_1)p'} ds \nonumber
    \\
    &\leq C d_x^{-\alpha \beta_0 p'/2} \int_0^{\infty } d_x^{\alpha}(\sqrt{s} + 1 )^{-\beta_0 p'-(1-\beta_1)p'} ds
    = C d_x^{\alpha-\alpha \beta_0 p'/2}.
\end{align}
Therefore, due to \eqref{21.06.18.1054}, \eqref{21.06.18.1053} and Fubini's theorem,
\begin{align} 
\nonumber
   &\int_0^T \int_{D} d_x^{\mu}|d_x^{-\alpha/2} \cT_{D} f(t,x)|^p dxdt \leq C \int_0^T \int_{D} d_x^{\mu+\alpha p/2 - \alpha -\alpha\beta_0 p/2} II(t,x)^p dxdt 
   \\
   &= C \int_0^T \int_{D}|f(s,y)|^p d_y^{\alpha \beta_0 p/2}   \label{eqn 10.7.4}\\
   &\qquad\qquad\times\left(\int_s^T \int_{D}  d_x^{\mu+\alpha p/2 - \alpha -\alpha\beta_0 p/2}  p(t-s,x-y)  R_{t-s,x}^{\beta_1 p} R_{t-s,y}^{\beta_2 p}  dxdt\right)dyds. \nonumber
\end{align}
Now, again by Lemma \ref{21.05.17.15.05} with $\gamma_0=2\mu/\alpha + p -2 -\beta_0 p + \beta_1 p$ and $\gamma_1=\beta_1 p$,
\begin{align} \label{21.08.10.1511}
    &\int_s^T \int_{D}  d_x^{\mu+\alpha p/2 - \alpha -\alpha\beta_0 p/2} p(t-s,x-y)  R_{t-s,x}^{\beta_1 p} R_{t-s,y}^{\beta_2 p}  dxdt \nonumber
    \\
    &\leq C d_y^{\alpha\beta_2 p/2} \int_s^T (\sqrt{t-s}+d_y^{\alpha/2})^{2\mu/\alpha + p -2 -\beta_0 p - \beta_2 p}dt  \nonumber
    \\
    &\leq C d_y^{\mu + \alpha p/2 - \alpha \beta_0 p/2} \int_0^{\infty}  (\sqrt{t}+1)^{2\mu/\alpha + p -2 -\beta_0 p - \beta_2 p} dt\leq C d_y^{\mu + \alpha p/2 - \alpha \beta_0 p/2}.
\end{align}
This and \eqref{eqn 10.7.4}  yield \eqref{21.08.10.1111}, and the lemma is proved.
\end{proof}

Next, we consider the operator $\cT^0_D$ defined for initial data.
\begin{lemma} \label{21.06.18.1404}
Let $\alpha\in(0,2)$ and $p\in(1,\infty)$. 
Suppose that
$$
d-1<\theta<d-1+p+\left(\alpha (p-1) \wedge \frac{3}{2}\alpha p\right).
$$
Then, there exists $C=C(d,\alpha,\theta,p,D)$ such that for any $u_0\in \psi^{-\alpha/2+\alpha/p}L_{p,\theta}(D)$,
\begin{align*}
\|\psi^{-\alpha/2}\cT_D^0u_0\|_{\bL_{p,\theta}(D,T)}\leq C\|\psi^{-\alpha/2+\alpha/p}u_0\|_{L_{p,\theta}(D)}.
\end{align*}
\end{lemma}

\begin{proof}
As in the proof of Lemma \ref{21.06.18.1403}, it is enough to prove
\begin{align*}
\int_0^T \int_D d_x^{\mu-\alpha p/2}|\cT_D^0 u_0 (t,x)|^p dxdt \leq C \int_D d_x^{\mu+\alpha-\alpha p/2}|u_0(x)|^p dx,
\end{align*}
where $\mu:=\theta-d$. Since $\mu\in(-1,p-1+\frac{3}{2}\alpha p)$, we can choose $\beta_0$ satisfying
\begin{equation*}
    \frac{2\mu}{p \alpha}-1-\frac{2}{p} < \beta_0 < \frac{2\mu}{p \alpha}-1+\frac{2}{p}+\frac{2}{p\alpha}
\end{equation*}
and
\begin{equation*}
    -\frac{2(p-1)}{p} = -\frac{2}{p'} <\beta_0<\left(2+\frac{2}{\alpha}\right) \frac{1}{p'}=\left(2+\frac{2}{\alpha}\right)\frac{p-1}{p},
\end{equation*}
where $p'=p/(p-1)$. Also, since  $\frac{2\mu}{p \alpha}-1+\frac{2}{p}< \frac{2}{p'\alpha} + 1$, we can choose $\beta_1$  satisfying
$$
\frac{2\mu}{p \alpha}-1+\frac{2}{p}<\beta_0+\beta_ 1< \frac{2}{p'\alpha} + 1.
$$
Let $R_{t,x}:=\frac{d_x^{\alpha/2}}{\sqrt{t}+d_x^{\alpha/2}}$. By Lemma \ref{21.06.15.1700} and H\"older's inequality,
\begin{align*}
|\cT_D^0 u_0(t,x)|\leq& C\left( \int_{D} p(t,x-y)d_y^{-\alpha \beta_0 p'/2} R_{t,y}^{(1-\beta_1)p'} dy \right)^{1/{p'}}
\\
&\times \left( \int_{D}  p(t,x-y) d_y^{\alpha\beta_0  p/2}R_{t,x}^p R_{t,y}^{\beta_1 p}|u_0(y)|^pdy\right)^{1/p}
\\
=:& C\times I(t,x)\times II(t,x),
\end{align*}
By Lemma \ref{21.05.17.15.05} with $\gamma_0=(1-\beta_1)p'-\beta_0 p' $ and $\gamma_1=(1-\beta_1)p'$, we have
\begin{align*}
    I(t,x)^{p'}&=\int_{D} p(t,x-y) d_y^{-\alpha\beta_0 p'/2} R_{t,y}^{(1-\beta_1)p'} dy\leq C(\sqrt{t}+d_x^{\alpha/2})^{-\beta_0 p'}
\end{align*}
Therefore, applying Fubini's theorem,
\begin{align} \label{21.06.18.1349}
   &\int_0^T \int_{D} d_x^{\mu}|d_x^{-\alpha/2} \cT_D^0 u_0(t,x)|^p  dxdt \nonumber
   \\
   &\leq C \int_0^T \int_{D} d_x^{\mu-\alpha p/2} (\sqrt{t}+d_x^{\alpha/2})^{-\beta_0 p} II(t,x)^p dxdt \nonumber
   \\
   &\leq C \int_{D}|u_0(y)|^p d_y^{\alpha\beta_0 p/2} K(T,y)dy,
\end{align}
where
\begin{align*}
    K(T,y)&:=\int_0^T R_{t,y}^{\beta_1p} \int_{D} p(t,x-y)  d_x^{\mu-\alpha p/2}(\sqrt{t}+d_x^{\alpha/2})^{-\beta_0 p} R_{t,x}^p  dxdt
    \\
    &=\int_0^T R_{t,y}^{\beta_1p} \int_{D} p(t,x-y) d_x^{\mu}(\sqrt{t}+d_x^{\alpha/2})^{-\beta_0 p-p}  dxdt.
\end{align*}
By Lemma \ref{21.05.17.15.05} with $\gamma_0=2\mu/\alpha$ and $\gamma_1=\beta_0 p + p$,
\begin{align*}
    K(T,y) &\leq C \int_0^T R_{t,y}^{\beta_1 p} (\sqrt{t}+d_y^{\alpha/2})^{2 \mu/\alpha -\beta_0 p - p} dt
    \\
    &\leq C d_y^{\mu - \alpha \beta_0 p / 2 -\alpha p /2+\alpha} \int_0^\infty (\sqrt{t}+1)^{2 \mu/\alpha -\beta_0 p - p -\beta_1 p}dt \\
    &\leq C d_y^{\mu - \alpha \beta_0 p / 2 -\alpha p /2+\alpha}.
\end{align*}
This with \eqref{21.06.18.1349}   proves  the lemma.
\end{proof}

Finally, we consider the operator $\cG^{\lambda}_D$ for  elliptic equation  \eqref{21.07.26.16.22}.
\begin{lemma}
\label{21.08.03.11.20}
Let $\alpha\in(0,2)$, $p\in(1,\infty)$ and $\theta\in(d-1,d-1+p)$.
Suppose that $D$ is a half space (resp. a bounded $C^{1,1}$ open set) and $\lambda>0$ (resp. $\lambda\geq0$).
Then, for any $f\in \psi^{-\alpha/2}L_{p,\theta}(D)$, 
\begin{align*}
    \|\psi^{-\alpha/2} \cG_D^{\lambda}f\|_{L_{p,\theta}(D)}\leq C\|\psi^{\alpha/2}f\|_{L_{p,\theta}(D)},
\end{align*}
where $C=C(d,p,\alpha,\theta,D)$ is independent of $\lambda$.
\end{lemma}

\begin{proof}
As before, we need to show  
\begin{align}
\label{21.08.10.18.18}
\int_D d_x^{\mu-\alpha p/2}|\cG_D^\lambda f(x)|^p dx \leq C \int_D d_x^{\mu+\alpha p/2}|f(x)|^p dx,
\end{align}
where $\mu:=\theta-d$. Take $\beta_0, \beta_1$ and $\beta_2$ satisfying \eqref{21.05.24.23.44}-\eqref{21.08.10.1110}.
Let $R_{t,x} := \frac{d_x^{\alpha/2}}{\sqrt{t}+d_x^{\alpha/2}}$. By Lemma \ref{21.06.15.1700} and H\"older's inequality,
\begin{align} \label{21.08.03.11.10}
|\cG_D^{\lambda} f(x)|\leq& C\left(\int_0^\infty \int_{D} p(t,x-y) d_y^{-\alpha \beta_0 p'/2} R_{t,x}^{(1-\beta_1)p'} R_{t,y}^{(1-\beta_2)p'} dydt \right)^{1/{p'}} \nonumber
\\
&\times \left(\int_0^\infty \int_{D}  p(t,x-y)d_y^{\alpha \beta_0 p/2}  R_{t,x}^{\beta_1 p} R_{t,y}^{\beta_2 p} |f(y)|^p dydt\right)^{1/p} \nonumber
\\
=:&C\times I(t,x)\times II(t,x).
\end{align}
Similar argument used to prove \eqref{21.06.18.1053} yields
\begin{align} \label{21.08.03.11.12}
    I(t,x)^{p'} \leq C d_x^{\alpha(1-\beta_1)p'/2} \int_0^\infty (\sqrt{t} + d_x^{\alpha/2} )^{-\beta_0 p'-(1-\beta_1)p'} dt \leq C d_x^{\alpha-\alpha \beta_0 p'/2}.
\end{align}
Therefore, by \eqref{21.08.03.11.10}, \eqref{21.08.03.11.12} and Fubini's theorem,
\begin{align*} 
   &\int_{D} d_x^{\mu}|d_x^{-\alpha/2} \cG_D^{\lambda} f(x)|^p dx
   \\
   &\leq C  \int_{D} d_x^{\mu+\alpha p/2 - \alpha -\alpha\beta_0 p/2} II(t,x)^p dxdt 
   \\
   &= C \int_{D}|f(y)|^p d_y^{\alpha \beta_0 p/2} \Big(\int_0^\infty \int_{D}  d_x^{\mu+\alpha p/2 - \alpha -\alpha\beta_0 p/2}  p(t,x-y)  R_{t,x}^{\beta_1 p} R_{t,y}^{\beta_2 p}  dxdt\Big)dy.
\end{align*}
As in \eqref{21.08.10.1511}, we get
\begin{align*}
    &\int_0^\infty \int_{D}  d_x^{\mu+\alpha p/2 - \alpha -\alpha\beta_0 p/2} p(t,x-y)  R_{t,x}^{\beta_1 p} R_{t,y}^{\beta_2 p}  dxdt
    \\
    &\leq C d_y^{\alpha\beta_2 p/2} \int_0^\infty (\sqrt{t}+d_y^{\alpha/2})^{2\mu/\alpha + p -2 -\beta_0 p - \beta_2 p}dt \leq C d_y^{\mu + \alpha p/2 - \alpha \beta_0 p/2}.
\end{align*}
Thus, we prove \eqref{21.08.10.18.18} and the lemma.
\end{proof}

\mysection{Higher order estimates} \label{21.09.20.1430}

In this section, we prove that  one can raise regularity of solutions as long as the free terms are in appropriate function spaces.

We first prepare some auxiliary results below. 
Let $\{\zeta_n : n\in \bZ\}$ be a collection of functions satisfying \eqref{21.10.03.18.06.1}-\eqref{21.10.03.18.06.3} with $(k_1,k_2)=(1,e^2)$.  We also take $\{\eta_n : n\in \bZ\}$ satisfying \eqref{21.10.03.18.06.1}-\eqref{21.10.03.18.06.3} with $(k_1,k_2)=(e^{-2},e^4)$ and
$$
\eta_n=1 \text{ on } \{x\in D : e^{-n-1}<\rho(x)< e^{-n+3}\}.
$$
Consequently, $\eta_n=1$ on the support of $\zeta_n$ and $\zeta_n \eta_n=\zeta_n$.

\begin{lemma}
\label{21.05.12.10.51}
For any  $\gamma\in\bR$, there exists a constant $C=C(d,\alpha,\gamma)$ such that for $u\in C_c^{\infty}(D)$ and $n\in\bZ$,
\begin{align*}
    &\left\|\Delta^{\alpha/2}\Big((u\zeta_{-n}\eta_{-n})(e^n\cdot)\Big)-
    \zeta_{-n}(e^n\cdot)\Delta^{\alpha/2}\Big((u\eta_{-n})(e^n\cdot) \Big)\right\|_{H_p^{\gamma}}  \\
    &\leq C\left(\left\|\Delta^{\alpha/4}\Big((u\eta_{-n})(e^n\cdot)\Big) \right\|_{H_p^{\gamma}}+\|u(e^n\cdot)\eta_{-n}(e^n\cdot)\|_{H_p^{\gamma}}\right).
\end{align*}
\end{lemma}

\begin{proof} 
By \eqref{21.07.20.1710},
    \begin{align} \nonumber
    &\Delta^{\alpha/2}\big((u\zeta_{-n}\eta_{-n})(e^n\cdot)\big)(x)-\zeta_{-n}(e^nx) \Delta^{\alpha/2}\big((u\eta_{-n})(e^n\cdot)\big)(x)\\
    &-u(e^n x)\eta_{-n}(e^nx)\Delta^{\alpha/2}\zeta_{-n}(e^n\cdot)(x) = C\int_{\bR^d}H_n(x,y)|y|^{-d-\alpha}dy,  \label{21.06.14.1525}
    \end{align}
where
$$
H_n(x,y):=[(u\eta_{-n})(e^n(x+y))-(u\eta_{-n})(e^nx)][\zeta_{-n}(e^n(x+y))-\zeta_{-n}(e^nx)].
$$
In the virtue of \eqref{21.10.03.18.06.2}, for any $m\in \bN_+$,
\begin{align*}
    \left|D^{m}_x\big(\zeta_{-n}(e^n(x+y))-\zeta_{-n}(e^nx)\big)\right| \leq C(m) (1\wedge|y|).
\end{align*}
Thus, $\zeta_{-n}(e^n(x+y))-\zeta_{-n}(e^nx)$ becomes a  point-wise multiplier in $H^{\gamma}_p$ (see e.g.  \cite[Lemma 5.2]{kry99analytic}), and  therefore
$$
\|H_n(\cdot,y)\|_{H_p^\gamma} \leq C (1\wedge|y|) \|(u\eta_{-n})(e^n(\cdot+y))-(u\eta_{-n})(e^n\cdot)\|_{H_p^\gamma}.
$$
By \cite[Lemma 2.1]{zhang2013lp}, the above is bounded by 
\begin{equation}
\label{21.07.24.13.15}
C\left(\|u(e^n\cdot)\eta_{-n}(e^n\cdot)\|_{H_p^{\gamma}} \wedge |y|^{\alpha/2+1}\|\Delta^{\alpha/4}\big(u(e^n\cdot)\eta_{-n}(e^n\cdot)\big)\|_{H_p^{\gamma}}\right).
\end{equation}
By Minkowski's inequality and \eqref{21.07.24.13.15},
\begin{align}  \label{21.06.14.1526}
    &\left\| \int_{\bR^d}H_n(\cdot,y)|y|^{-d-\alpha}dy \right\|_{H_p^{\gamma}} \nonumber\\
    &\leq C \|\Delta^{\alpha/4}\big(u(e^n\cdot)\eta_{-n}(e^n\cdot)\big)\|_{H_p^{\gamma}} \int_{|y|\leq1}|y|^{-d-\alpha/2+1}dy \nonumber
    \\
    &\quad+ C \|u(e^n\cdot)\eta_{-n}(e^n\cdot)\|_{H_p^{\gamma}} \int_{|y|>1}|y|^{-d-\alpha}dy \nonumber
    \\
    &\leq C\left(\|\Delta^{\alpha/4}\big((u\eta_{-n})(e^n\cdot)\big)\|_{H_p^{\gamma}}+\|u(e^n\cdot)\eta_{-n}(e^n\cdot)\|_{H_p^{\gamma}}\right). 
\end{align}
On the other hand, by \eqref{21.07.20.1710} and \eqref{21.10.03.18.06.2},
\begin{align*}
\label{21.06.14.17.45}
    |D^m_x \Delta^{\alpha/2}(\zeta_{-n}(e^n\cdot))(x)| &\leq C \|D^{m+2}_x \zeta_{-n}(e^n\cdot)\|_{L_{\infty}}\int_{|y|\leq 1}|y|^{-d-\alpha+2}dy
    \\
    &\quad +  C \|D^m_x\zeta_{-n}(e^n\cdot)\|_{L_{\infty}}\int_{|y|> 1}|y|^{-d-\alpha}dy \leq C.
\end{align*}
Thus, again by \cite[Lemma 5.2]{kry99analytic}, we have
\begin{equation}
\label{21.06.14.17.45}
    \|u(e^n \cdot)\eta_{-n}(e^n\cdot)\Delta^{\alpha/2}(\zeta_{-n}(e^n\cdot))\|_{H_p^{\gamma}}\leq C\|u(e^n \cdot)\eta_{-n}(e^n\cdot)\|_{H_p^{\gamma}}.
\end{equation}
Combining \eqref{21.06.14.1525}, \eqref{21.06.14.1526} and \eqref{21.06.14.17.45},  we prove the lemma.
\end{proof}

\begin{lemma}
\label{21.05.12.10.49}
Let  
$d-1-\alpha p/2<\theta<d-1+p+\alpha p/2$.
Then,  for any $\gamma \in \bR$ and $u\in  C_c^{\infty}(D)$,
\begin{align} \label{21.06.14.1542}
&\sum_{n\in\bZ} e^{n(\theta-\alpha p/2)} \left\|\zeta_{-n}(e^n\cdot)\Delta^{\alpha/2}\Big( [1-\eta_{-n}(e^n\cdot)]u(e^n\cdot) \Big) \right\|_{H_p^\gamma}^p\leq C\|\psi^{-\alpha/2}u\|_{L_{p,\theta}(D)}^p.
\end{align}
where $C=C(d,p,\gamma,\alpha,\theta,D)$.
\end{lemma}

\begin{proof}
It is certainly enough  to prove \eqref{21.06.14.1542} for only $\gamma=m \in\bN_+$. 

By the choice of $\{\eta_{-n}:n\in\bZ\}$, we have $\zeta_n(x)=\zeta_n(x)\eta_n(x)$ for all $x$, and 
\begin{equation*}
\zeta_{-n}(e^n x)(1-\eta_{-n}(e^n(x+y)))=0 \quad \text{if} \quad  |y|< \delta_0,
\end{equation*}
where $\delta_0:=1-e^{-1}$.  Thus, by \eqref{21.07.20.1710}, 
    \begin{align}  \label{21.06.14.1539}
F_n(x)&:=\zeta_{-n}(e^nx)\Delta^{\alpha/2}\Big( [1-\eta_{-n}(e^n\cdot)]u(e^n\cdot)\Big)(x)   \nonumber
\\
&=C\int_{|y|\geq \delta_0} u(e^n(x+y))\zeta_{-n}(e^nx) [1-\eta_{-n}(e^n(x+y))] |y|^{-d-\alpha}dy \nonumber
\\
&=C\int_{|x-y|\geq \delta_0} u(e^ny) \Big(\zeta_{-n}(e^nx)[1-\eta_{-n}(e^ny)] |x-y|^{-d-\alpha} \Big)dy.
\end{align}
Denote  
$$B_n:=supp(\zeta_{-n}).
$$ Then, since $\zeta_{-n}(e^n x)(1-\eta_{-n}(e^ny))=0$ for $|x-y| < \delta_0$, by \eqref{21.10.03.18.06.2},  we have
\begin{align*}
&\left|D_x \Big(\zeta_{-n}(e^n x)(1-\eta_{-n}(e^n y))|x-y|^{-d-\alpha}\Big) \right| \\
&\leq C 1_{B_n}(e^n x) |1-\eta_{-n}(e^ny)||x-y|^{-d-\alpha}
\\
&\quad+ C |\zeta_{-n}(e^n x)(1-\eta_{-n}(e^n y))||x-y|^{-d-\alpha-1}
\\
&\leq C 1_{B_n}(e^n x) |1-\eta_{-n}(e^ny)||x-y|^{-d-\alpha}.
\end{align*}
Similarly, for $k\in \bN_+$,
\begin{align*}
&\left|D_x^k \Big(\zeta_{-n}(e^n x)(1-\eta_{-n}(e^ny))|x-y|^{-d-\alpha}\Big) \right| \\
&\leq C(k) 1_{B_n}(e^n x) |1-\eta_{-n}(e^ny)||x-y|^{-d-\alpha}.
\end{align*}
It follows from \eqref{21.06.14.1539} for each $k\in \bN_+$,
\begin{equation}
  \label{eqn 10.8.7}
|D^k_x F_n (x)|\leq C(k) H_n(x),
\end{equation}
where 
\begin{align*}
    H_n(x):= 1_{B_n}(e^n x) \int_{|x-y|\geq \delta_0} |u(e^ny)| \left|1-\eta_{-n}(e^ny)\right| |x-y|^{-d-\alpha}dy.
\end{align*}
Since $\|F_n\|_{H^m_p} \approx \sum_{k\leq m} \|D^k_x F_n\|_{L_p}$, from \eqref{eqn 10.8.7} we get
\begin{align*}
&\sum_{n\in\bZ} e^{n(\theta-\alpha p/2)}\|\zeta_{-n}(e^n\cdot)\Delta^{\alpha/2}\big([1-\eta_{-n}(e^n\cdot)]u(e^n\cdot)\big)\|_{H_p^m}^p
\\
&\leq C\sum_{n\in\bZ} e^{n(\theta-\alpha p/2)} \|H_n\|^p_{L_p}.
\end{align*}
Therefore, to finish the proof of  \eqref{21.06.14.1542}, we only need to show
\begin{align} \label{21.08.16.2002}
    \sum_{n\in\bZ} e^{n(\theta-\alpha p/2)} \|H_n\|_{L_p}^p \leq C\|\psi^{-\alpha/2}u\|_{L_{p,\theta}(D)}^p.
\end{align}

Case 1. Let $d-1+\alpha p/2<\theta<d-1+p+\alpha p/2$. Observe that
\begin{align*}
    &\int_{|y|\geq \delta_0} |u(e^n(x+y))(1-\eta_{-n}(e^n(x+y)))||y|^{-d-\alpha}dy    
    \\
    &\leq\sum_{k=0}^{\infty}\int_{2^{k}\delta_0\leq |y|<2^{k+1}\delta_0} |u(e^n(x+y))||y|^{-d-\alpha}dy
    \\
    &= C(d) e^{n\alpha}\sum_{k=0}^{\infty}\int_{2^{k}e^n\delta_0\leq |y|<2^{k+1}e^n\delta_0} |u(e^nx+y)||y|^{-d-\alpha}dy
    \\
    &\leq C \sum_{k=0}^{\infty}2^{-k\alpha}\frac{1}{e^{nd}2^{kd}}\int_{2^{k}e^n\delta_0\leq |y|<2^{k+1}e^n\delta_0} |u(e^nx+y)|dy
    \\
    &\leq C  \sum_{k=0}^{\infty}2^{-k\alpha} \bM u(e^n x) =C \bM u(e^n x),
\end{align*}
where $\bM u$ is the maximal function of $u$ defined by
\begin{align*}
\bM u(x)=\sup_{x\in B_r(z)}\frac{1}{|B_r(z)|}\int_{B_r(z)}|u(y)|dy.
\end{align*}
Therefore, $H_n(x)\leq C (1_{B_n}\bM u) (e^nx)$. Since $e^n \approx \rho$ on $B_n$, by the change of variables,
\begin{align*}
     \sum_{n\in\bZ} e^{n(\theta-\alpha p/2)} \|H_n\|_{L_p}^p &\leq C\int_{D}|\bM u(x)|^p \rho(x)^{\theta-d-\alpha p/2}dx.
\end{align*}
Due to   \cite[Theorem 1.1]{dyda2019muckenhoupt}, the function $\rho^{\theta-\alpha p/2-d}$ belongs to the class of Muckenhoupt $A_p$-weights,  and therefore we can apply  the Hardy-Littlewood  Maximal inequality (\cite[Theorem 7.1.9]{grafakos2014classical})  to get
\begin{align*}
 \sum_{n\in\bZ} e^{n(\theta-\alpha p/2)} \|H_n\|_{L_p}^p \leq C\int_{\bR^d}|u(x)|^p\rho(x)^{\theta-d-\alpha p/2}dx.
\end{align*}
This proves \eqref{21.08.16.2002} if $d-1+\alpha p/2<\theta<d-1+p+\alpha p/2$.

\vspace{1mm}

Case 2. Let $d-1-\alpha p/2<\theta<d+\alpha p/2$. Then, we can choose $\beta\in(0,\alpha)$ such that 
$$
-1<\theta-d-\alpha p/2+\beta p\leq 0.
$$
By \eqref{21.06.14.1539} and H\"older's inequality, for $p':=p/(p-1)$,
\begin{align} \label{21.06.14.1541}
H_n(x) &\leq 1_{B_n}(e^n x)\left(\int_{|x-y|\geq \delta_0} \frac{|u(e^ny)(1-\eta_{-n}(e^ny))|^p}{|x-y|^{d+\beta p}}dy\right)^{1/p}\nonumber\\
& \qquad\qquad\qquad\times\left(\int_{|x-y|\geq \delta_0} |x-y|^{-d-(\alpha-\beta)p'}dy\right)^{1/p'} \nonumber
\\
& \leq  C 1_{B_n}(e^n x) \left(\int_{|x-y|\geq \delta_0}  \frac{|u(e^ny)|^p}{|x-y|^{d+\beta p}}dy\right)^{1/p}.
\end{align}    
By the change of variables  and Fubini's theorem,
\begin{align} \label{21.06.14.17.55}
&\sum_{n\in\bZ} e^{n(\theta-\alpha p/2)} \|H_n\|_{L_p}^p\nonumber\\
&\leq\sum_{n\in\bZ}e^{n(\theta-\alpha p/2)}\int_{\bR^d}\int_{|x-y|\geq \delta_0}  1_{B_n}(e^n x) \frac{|u(e^ny)|^p}{|x-y|^{d+\beta p}}dydx \nonumber
\\
&= \int_{\bR^d} \left[\sum_{n\in\bZ}e^{n(\theta-d-\alpha p/2+\beta p)}\int_{|x-y|\geq e^n\delta_0}  1_{B_n}(x) |x-y|^{-d-\beta p}dx\right] |u(y)|^pdy.
\end{align}

In the virtue of \eqref{21.06.14.1541} and \eqref{21.06.14.17.55}, to prove \eqref{21.08.16.2002}, it suffices to show that for $y\in D$,
\begin{equation}
    \label{21.06.14.19.11}
    \begin{aligned}
    \sum_{n\in\bZ}e^{n(\theta-d-\alpha p/2+\beta p)}\int_{|x-y|\geq e^n\delta_0}  1_{B_n}(x) |x-y|^{-d-\beta p}dx \leq C d_y^{\theta-d-\alpha p/2}.
\end{aligned}
\end{equation}
For  fixed $y\in D$, we take $n_0=n_0(y)\in\bZ$ such that
$$
e^{n_0+3}\leq d_y<e^{n_0+4}.
$$

If $n\leq n_0$ and $x\in B_n$, then  $e^n<d_x<e^{n+2}\leq e^{n_0+2} < e^{n_0+3} \leq d_y$, and consequently $|x-y|\geq d_y-d_x \geq C e^{n_0}$. 
Thus,
\begin{align} \label{21.06.14.19.09}
 &\sum_{n\leq n_0}e^{n(\theta-d-\alpha p/2+\beta p)}\int_{|x-y|\geq Ce^n} 1_{B_n}(x) |x-y|^{-d-\beta p}dx \nonumber\\
 &\leq C \int_{|x-y|\geq Ce^{n_0}}  \sum_{n\leq n_0} 1_{B_n}(x) \frac{d_x^{\theta-d-\alpha p/2+\beta p}}{|x-y|^{d+\beta p}}dx \nonumber
\\
&\leq  C\int_{|x-y|\geq Ce^{n_0}, d_x< d_y} \frac{d_x^{\theta-d-\alpha p/2+\beta p}}{|x-y|^{d+\beta p}}dx \nonumber
\\
&\leq Ce^{-n_0\beta p}d_y^{\theta-d-\alpha p/2+\beta p} \leq Cd_y^{\theta-d-\alpha p/2},
\end{align}
where $C$ is independent of $y$. For the second inequality above, we used $\underset{n\leq n_0}{\sum}1_{B_n}(x)\leq C1_{d_x<d_y}$, and for the third inequality, we used Lemma \ref{21.05.13.11.18}($ii$) with $\rho=Ce^{n_0}$ and $r=d_y$.

Next, we handle the summation for $n>n_0$. Since $\theta-\alpha p/2-d<0$,
    \begin{align} \label{21.06.14.19.10}
    &\sum_{n>n_0}e^{n(\theta-d-\alpha p/2+\beta p)}\int_{|x-y|\geq \delta_0e^n} 1_{B_n}(x) |x-y|^{-d-\beta p}dx \nonumber\\
    &\leq C \sum_{n>n_0}e^{n(\theta-d-\alpha p/2+\beta p)}\int_{|x-y|\geq\delta_0 e^{n}}|x-y|^{-d-\beta p}dx \nonumber
    \\
    &=C\sum_{n>n_0}e^{n(\theta-d-\alpha p/2)} = Ce^{n_0(\theta-d-\alpha p/2)} \leq C d_y^{\theta-d-\alpha p/2}.
\end{align}
Combining \eqref{21.06.14.19.09} and \eqref{21.06.14.19.10}, we obtain \eqref{21.06.14.19.11}. Thus, \eqref{21.08.16.2002} and the lemma are proved.
\end{proof}

\begin{lemma} \label{21.06.14.1610}
Let 
$d-1-\alpha p/2<\theta<d-1+p+\alpha p/2$ and $\gamma\in\bR$.  Then, for any $u \in  C^{\infty}_c(D)$, 
\begin{align} 
\label{eqn 10.9.5}
&\sum_{n\in\bZ} e^{n(\theta-\alpha p/2)} \left\|\Delta^{\alpha/2}\Big(u(e^n\cdot)\zeta_{-n}(e^n\cdot)\Big)-\zeta_{-n}(e^n\cdot)\Delta^{\alpha/2}(u(e^n\cdot)) \right\|_{H_p^{\gamma}}^p\nonumber
\\
&\leq C \|\psi^{-\alpha /2}u\|_{H_{p,\theta}^{0\vee(\gamma+\alpha/2)}(D)}^p,
\end{align}
where $C=C(d,p,\alpha,\theta, \gamma, D)$.  
\end{lemma}
\begin{proof}
Recall $\eta_{-n}\zeta_{-n}=\zeta_{-n}$. Thus, by the triangle inequality,
\begin{eqnarray*}
&& \left\|\Delta^{\alpha/2}\Big(u(e^n\cdot)\zeta_{-n}(e^n\cdot)\Big)-\zeta_{-n}(e^n\cdot)\Delta^{\alpha/2}(u(e^n\cdot)) \right\|_{H_p^{\gamma}}\\
&&\leq  \left\|\Delta^{\alpha/2}\Big((u\zeta_{-n}\eta_{-n})(e^n\cdot)\Big)-
    \zeta_{-n}(e^n\cdot)\Delta^{\alpha/2}\Big((u\eta_{-n})(e^n\cdot) \Big)\right\|_{H_p^{\gamma}}\\
    &&\quad + \left\|\zeta_{-n}(e^n\cdot)\Delta^{\alpha/2}\Big( [1-\eta_{-n}(e^n\cdot)]u(e^n\cdot) \Big) \right\|_{H_p^\gamma}.
         \end{eqnarray*}
Also, note
$$
\|u\|_{H^{\gamma+\alpha/2}_p} \approx \left( \|u\|_{H^{\gamma}_{p}}+\|\Delta^{\alpha/4}u\|_{H^{\gamma}_p} \right).
$$
Therefore, Lemma \ref{21.05.12.10.51} and  Lemma \ref{21.05.12.10.49} easily lead to the claim of the lemma.
\end{proof}

\begin{lemma}
\label{cor 10.10}
Let $d-1-\alpha p/2<\theta<d-1+p+\alpha p/2$, and $\gamma\geq -\alpha$. Then,
for any $u \in  C^{\infty}_c(D)$,  we have $\Delta^{\alpha/2} u \in  \psi^{-\alpha/2}H^{\gamma}_{p,\theta}(D)$ and
\begin{equation}
 \label{eqn re}
\|\psi^{\alpha/2}\Delta^{\alpha/2}u\|_{H^{\gamma}_{p,\theta}(D)}\leq C \|\psi^{-\alpha/2}u\|_{H^{\gamma+\alpha}_{p,\theta}(D)}.
\end{equation}
\end{lemma}

\begin{proof}
By Lemma \ref{21.06.15.13.47}($ii$) and the relation $\Delta^{\alpha/2}u(e^nx)=e^{-n\alpha} (\Delta^{\alpha/2} u(e^n\cdot))(x)$, 
\begin{eqnarray*}
\|\psi^{\alpha/2}\Delta^{\alpha/2}u\|_{H^{\gamma}_{p,\theta}(D)}^p &\leq&
 C \sum_n e^{n(\theta+\alpha p/2)} \|\zeta_{-n}(e^n\cdot) \Delta^{\alpha/2}u (e^n\cdot)\|^p_{H^{\gamma}_p}
 \\
 &=&C  \sum_n e^{n(\theta-\alpha p/2)} \|\zeta_{-n}(e^n\cdot) \Delta^{\alpha/2}(u (e^n\cdot))\|^p_{H^{\gamma}_p}.
 \end{eqnarray*}
 By \eqref{eqn 10.9.5}, the last term above is bounded by
 \begin{eqnarray*}
 &&C\sum_{n\in\bZ} e^{n(\theta-\alpha p/2)} \left\|\Delta^{\alpha/2}\Big(u(e^n\cdot)\zeta_{-n}(e^n\cdot)\Big)\right\|_{H_p^{\gamma}}^p +
   C\|\psi^{-\alpha/2}u\|_{H_{p,\theta}^{0\vee(\gamma+\alpha/2)}(D)}^p 
  \\
  &&\quad \leq C \|\psi^{-\alpha/2}u\|_{H_{p,\theta}^{\gamma+\alpha}(D)}^p.
  \end{eqnarray*}
  The lemma is proved.
  \end{proof}

  By Lemma \ref{cor 10.10},  for any $\gamma_0\in\bR$ and $\phi\in C^{\infty}_c(D)$,   $\Delta^{\alpha/2}\phi$ belongs to the dual space of 
  $H^{\gamma_0+\alpha}_{p,\theta-\alpha p/2}(D)$ (see Lemma \ref{21.06.15.13.47}($iii$)).  Therefore, for $u\in \psi^{\alpha/2}H^{\gamma_0+\alpha}_{p,\theta}(D)$,  we can define $\Delta^{\alpha/2} u$ as a distribution on $D$ by
\begin{align} \label{21.10.21.2139}
(\Delta^{\alpha/2} u, \phi)_D:=(u, \Delta^{\alpha/2}\phi)_{D}, \quad \phi\in C^{\infty}_c(D).
\end{align}
  
  \begin{corollary}
\label{cor 10.10-1}
Let $d-1-\alpha p/2<\theta<d-1+p+\alpha p/2$. 

(i) Let $\gamma\in\bR$, $u\in \psi^{\alpha/2}H^{\gamma+\alpha}_{p,\theta}(D)$, and  $\Delta^{\alpha/2} u$ be defined as in \eqref{21.10.21.2139}. 
 Then, $\Delta^{\alpha/2} u \in  \psi^{-\alpha/2}H^{\gamma}_{p,\theta}(D)$, and  \eqref{eqn re} holds.
 
 (ii) If $\gamma\geq -\alpha/2$ and $u\in H^{\gamma+\alpha/2}_{p,\theta-\alpha p/2}(D)$, then the left-hand side of  \eqref{eqn 10.9.5} makes sense, and inequality  \eqref{eqn 10.9.5} holds.
  \end{corollary}

\begin{proof}
If $\gamma\geq0$, then ($i$) is a consequence of Lemma  \ref{cor 10.10} and Lemma \ref{21.06.15.13.47}($i$). If $\gamma<0$, then by Lemmas \ref{cor 10.10} and \ref{21.06.15.13.47}$(iii)$,
\begin{align*} 
|(\Delta^{\alpha/2} u, \phi)_D|&\leq C \|\psi^{-\alpha/2}u\|_{H_{p,\theta}^{\gamma+\alpha}(D)}\|\psi^{\alpha/2}\Delta^{\alpha/2}\phi\|_{H_{p',\theta'}^{-\gamma-\alpha}(D)}
\\
&\leq C\|\psi^{-\alpha/2}u\|_{H_{p,\theta}^{\gamma+\alpha}(D)}\|\psi^{-\alpha/2}\phi\|_{H_{p',\theta'}^{-\gamma}(D)},
\end{align*}
where $1/p+1/p'=1$ and $\theta/p+\theta'/p'=d$. This implies $\Delta^{\alpha/2}$ is a bounded linear operator from $\psi^{\alpha/2}H_{p,\theta}^{\gamma+\alpha}(D)$ to $\psi^{-\alpha/2}H_{p,\theta}^\gamma(D)$.  Thus, $(i)$ is proved.

Next, we show $(ii)$. The left-hand side of \eqref{eqn 10.9.5} makes sense due to ($i$) and  \eqref{distribution}. Now, the claim of ($ii$) follows from Lemma \ref{21.06.14.1610}  and Lemma \ref{21.06.15.13.47}($i$). The corollary is proved.
\end{proof}

\begin{theorem}[Higher regularity for parabolic equation] \label{21.06.18.1433}
Let  $0\leq \mu\leq \gamma$, and $\theta \in (d-1-\frac{\alpha p}{2},d-1+p+\frac{\alpha p}{2})$. Suppose that 
$f\in\psi^{-\alpha/2}\bH_{p,\theta}^{\gamma-\alpha}(D,T)$,  $u_0\in\psi^{\alpha/2-\alpha/p}B_{p,\theta}^{\gamma-\alpha/p}(D)$, and 
 $u\in \psi^{\alpha/2}{\bH}^{\mu}_{p,\theta}(D,T) \cap \{u=0 \,\text{on}\, [0,T]\times D^c\}$ is a  weak solution to \eqref{21.06.18.1430}. Then,
 $u\in \psi^{\alpha/2}\bH_{p,\theta}^{\gamma}(D,T)$, and for this solution 
\begin{align} \label{21.07.28.1526}
\nonumber
 \|\psi^{-\alpha/2}u\|_{\bH_{p,\theta}^{\gamma}(D,T)}
 &\leq C\Big(\|\psi^{-\alpha/2+\alpha/p}u_0\|_{B_{p,\theta}^{\gamma-\alpha/p}(D)}\\
 &\quad\quad +\|\psi^{\alpha/2}f\|_{\bH_{p,\theta}^{\gamma-\alpha}(D,T)}+\|\psi^{-\alpha/2}u\|_{\bH^{\mu}_{p,\theta}(D,T)}   \Big),
\end{align}
where $C=C(d,p,\alpha,\gamma,\mu,\theta,D)$.
\end{theorem}

\begin{proof}
\textbf{1}. We first note that it is enough to consider the case $\gamma \leq \mu+\alpha/2$.  Indeed, if the claim holds for the case $\gamma \leq \mu+\alpha/2$, then repeating the result with $\mu'=\mu+\alpha/2, \mu+2\alpha/2, \cdots$ in order, we prove the lemma when $\gamma=\mu+k \alpha/2$, $k\in \bN_+$. Now let 
$\gamma=\mu+k \alpha/2+c$, where $k\in \bN_+$ and $c\in (0,\alpha/2)$.  Then,  applying the previous result with $\mu'=\mu+k\alpha/2$, we prove the general case.

\vspace{1mm}

\textbf{2}. 
For each $n\in\bZ$, denote 
$$u_n(t,x):=u(e^{n\alpha}t, e^n x), \quad f_n(t,x):=f(e^{n\alpha}t, e^n x), \quad u_{0n}(x):=u_0(e^nx).
$$
 Then, $u_n(\cdot){\zeta}_{-n}(e^n\cdot) \in \bH^{\mu}_{p}(e^{-n\alpha}T)$ and it is a weak solution (or solution in the sense of distribution) to the equation
\begin{equation*}
\begin{cases}
\partial_t v_n(t,x)=\Delta^{\alpha/2} v_n(t,x)+F_n(t,x),\quad &(t,x)\in(0,e^{-n\alpha}T)\times \bR^d
\\
v_n(0,x)=(u_{0n}(\cdot)\zeta_{-n}(e^n\cdot))(x),\quad & x\in \bR^d
\end{cases}
\end{equation*}
where
\begin{align*}
F_n(t,x)&=e^{n\alpha}(f_n(\cdot,\cdot)\zeta_{-n}(e^n\cdot))(t,x)
\\
&\quad - \left(\Delta^{\alpha/2}(u_n(\cdot,\cdot)\zeta_{-n}(e^n\cdot))(t,x)-\zeta_{-n}(e^nx) \Delta^{\alpha/2}u_n(t,x)\right)
\\
&=:e^{n\alpha}(f_n(\cdot,\cdot)\zeta_{-n}(e^n\cdot))(t,x)-G_n(t,x).
\end{align*}

\vspace{1mm}

\textbf{3}.  By Corollary \ref{cor 10.10-1}($ii$) with $\gamma'=\mu-\alpha/2$, we have 
\begin{align}
\label{21.07.24.15.13}
    \sum_{n\in\bZ}e^{n(\theta-\alpha p/2)}\|G_n(e^{-n\alpha}t,\cdot)\|_{H_p^{\gamma-\alpha}}^p &\leq C\sum_{n\in\bZ}e^{n(\theta-\alpha p/2)}\|G_n(e^{-n\alpha}t,\cdot)\|_{H_p^{\mu-\alpha/2}}^p \nonumber
    \\
    &\leq C\|\psi^{-\alpha/2}u(t,\cdot)\|_{H_{p,\theta}^{\mu}(D)}^p.
\end{align}
Therefore, due to $f \in \psi^{-\alpha/2}\bH_{p,\theta}^{\gamma-\alpha}(D,T)$,
$$
F_n\in\bH_{p}^{\gamma-\alpha}(e^{-n\alpha}T).
$$
Thus, we apply  \cite[Theorem 1]{mikulevivcius2019cauchy} to conclude $u_n\zeta_{-n}(e^n \cdot)\in \bH^{\gamma}_{p}(e^{-n\alpha}T)$,  and
\begin{align} \label{21.07.20.2011}
&\|\Delta^{\alpha/2}(u(\cdot,e^n\cdot)\zeta_{-n}(e^n\cdot))\|_{\bH_p^{\gamma-\alpha}(T)}^p \nonumber
\\
&=e^{n\alpha}\|\Delta^{\alpha/2}(u_n(\cdot,\cdot)\zeta_{-n}(e^n\cdot))\|_{\bH_p^{\gamma-\alpha}(e^{-n\alpha}T)}^p  \nonumber
\\
&\leq Ce^{n\alpha}\|\zeta_{-n}(e^n\cdot)u_{0n}(\cdot)\|_{B_{p}^{\gamma-\alpha/p}}^p
+Ce^{n\alpha}\|G_{n}(\cdot,\cdot)\|_{\bH_p^{\gamma-\alpha}( e^{-n\alpha}T)}^p   \nonumber 
\\
&\quad+Ce^{n\alpha}\|e^{n\alpha}\zeta_{-n}(e^n\cdot) f_n(\cdot,\cdot)\|_{\bH_p^{\gamma-\alpha}(e^{-n\alpha}T)}^p \nonumber 
\\
&=Ce^{n\alpha}\|\zeta_{-n}(e^n\cdot)u_{0n}(\cdot)\|_{B_{p}^{\gamma-\alpha/p}}^p
+C\|G_{n}(e^{-n\alpha}\cdot,\cdot)\|_{\bH_p^{\gamma-\alpha}(T)}^p   \nonumber 
\\
&\quad+C\|e^{n\alpha}\zeta_{-n}(e^n\cdot) f(\cdot,e^n\cdot)\|_{\bH_p^{\gamma-\alpha}(T)}^p.
\end{align}
By \eqref{21.07.24.15.13} and \eqref{21.07.20.2011} (also see Lemma \ref{21.06.15.13.47}($ii$)),
\begin{eqnarray}
\label{final}
&& \sum_{n\in\bZ}e^{n(\theta-\alpha p/2)}\|\Delta^{\alpha/2}(u(\cdot,e^n\cdot)\zeta_{-n}(e^n\cdot))\|_{\bH_p^{\gamma-\alpha}(T)}^p \nonumber
\\
 &&\leq C\Big(\|\psi^{-\alpha/2+\alpha/p}u_0\|_{B_{p,\theta}^{\gamma-\alpha/p}(D)}^p \nonumber
 \\
 &&\quad+\|\psi^{\alpha/2}f\|_{\bH_{p,\theta}^{\gamma-\alpha}(D,T)}^p+\|\psi^{-\alpha/2}u\|_{\bH_{p,\theta}^{\mu}(D,T)}^p\Big). 
 \end{eqnarray}
Therefore, \eqref{final}, Lemma \ref{21.06.15.13.47}($ii$),  and  the relation 
 $$\|u\|_{H^{\gamma}_{p}} \approx \left(\|u\|_{H_p^{\gamma-\alpha}}+\|\Delta^{\alpha/2}u\|_{H_p^{\gamma-\alpha}}\right)$$
 yield 
  \eqref{21.07.28.1526} for $\gamma \leq \mu+\alpha/2$.  The theorem is proved. 
\end{proof}

\begin{theorem}[Higher regularity for elliptic equation] \label{21.07.13.13.38}
Let $\lambda\geq 0$, $0\leq \mu\leq \gamma$,  and $\theta\in (d-1-\frac{\alpha p}{2},d-1+p+\frac{\alpha p}{2})$.
Suppose that $f\in \psi^{-\alpha/2}H_{p,\theta}^{\gamma-\alpha}(D)$, and $u\in \psi^{\alpha/2}H^{\mu}_{p,\theta}(D)\cap \{u=0\,\text{on}\, D^c\}$ is a solution to \eqref{21.07.13.13.02}, then, $u\in \psi^{\alpha/2}H_{p,\theta}^{\gamma}(D)$, and moreover
\begin{align*}
&\lambda \| \psi^{\alpha/2}u\|_{H_{p,\theta}^{\gamma-\alpha}(D)}+ \|\psi^{-\alpha/2}u\|_{H_{p,\theta}^{\gamma}(D)} 
\\
&\leq C\left(\|\psi^{\alpha/2}f\|_{H_{p,\theta}^{\gamma-\alpha}(D)}+\|\psi^{-\alpha/2}u\|_{H^{\mu}_{p,\theta}(D)}\right),
\end{align*}
where $C=C(d,p,\alpha,\gamma,\mu, \theta,D)$. In particular, $C$ is independent of $\lambda$.
\end{theorem}

\begin{proof}
We repeat the argument of the proof of Theorem \ref{21.06.18.1433}.  As before, we may assume $\gamma \leq \mu+\alpha/2$.

 Let $n\in\bZ$. Since $u$ is a weak solution to \eqref{21.07.13.13.02}, $u_n(x):=u(e^n x)$ and $f_n(x):=f(e^nx)$  satisfy the following equation in weak sense;
\begin{equation}
\label{21.07.13.13.43}
\Delta^{\alpha/2}(u_n(\cdot)\zeta_{-n}(e^n\cdot))(x)-e^{n\alpha}\lambda(u_n(\cdot)\zeta_{-n}(e^n\cdot))(x)= F_n(x),\quad x\in \bR^d,
\end{equation}
where
\begin{align*}
F_n(x)&=e^{n\alpha}f_n(x)\zeta_{-n}(e^nx)-G_n(x)
\\
&:=e^{n\alpha}f_n(x)\zeta_{-n}(e^nx)- \left(\Delta^{\alpha/2}(u_n(\cdot)\zeta_{-n}(e^n\cdot))(x)-\zeta_{-n}(e^nx) \Delta^{\alpha/2}u_n(x)\right). 
\end{align*}
   By Corollary \ref{cor 10.10-1}($ii$), we get $G_n\in H_{p}^{\gamma-\alpha}$ and
\begin{align}
\label{21.07.26.15.53}
    \sum_{n\in\bZ}e^{n(\theta-\alpha p/2)}\|G_n\|_{H_p^{\gamma-\alpha}}^p &\leq C\sum_{n\in\bZ}e^{n(\theta-\alpha p/2)}\|G_n\|_{H_p^{\mu-\alpha/2}}^p \nonumber
    \\
    &\leq C\|\psi^{-\alpha/2}u\|_{H_{p,\theta}^{\mu}(D)}^p.
\end{align}
This implies that $F_n\in H_{p}^{\gamma-\alpha}$.

If $\lambda=0$, then the equality \eqref{21.07.13.13.43} easily yields
\begin{align} \label{21.07.26.15.42}
    \|\Delta^{\alpha/2}(u_n(\cdot)\zeta_{-n}(e^n\cdot))\|^p_{H_p^{\gamma-\alpha}} &=\|F_n\|^p_{H_p^{\gamma-\alpha}} \nonumber
    \\
    &\leq  \|e^{n\alpha}f_n(\cdot)\zeta_{-n}(e^n\cdot)\|^p_{H_p^{\gamma-\alpha}}+\|G_n\|_{H_p^{\gamma-\alpha}}^p.
\end{align}

Next, let $\lambda>0$.  Then, by \cite[Theorem 1]{mikulevivcius2017p} (or \cite[Theorem 2.1]{dong2012lp}), we have $u_n(\cdot)\zeta_{-n}(e^n\cdot) \in H_p^{\gamma-\alpha}$ and
\begin{align}
\label{21.07.26.15.43}
    &e^{n\alpha p}\lambda^p \|u_n(\cdot)\zeta_{-n}(e^n\cdot) \|^p_{H_p^{\gamma-\alpha}} + \|\Delta^{\alpha/2}( u_n(\cdot)\zeta_{-n}(e^n\cdot))\|^p_{H_p^{\gamma-\alpha}} \nonumber
    \\
    &\leq C\| F_n\|^p_{H_p^{\gamma-\alpha}} 
    \leq C\left( \|e^{n\alpha}f_n(\cdot)\zeta_{-n}(e^n\cdot)\|^p_{H_p^{\gamma-\alpha}}+\|G_n\|^p_{H_p^{\gamma-\alpha}}\right).
\end{align}
We multiply by $e^{n(\theta-\alpha p/2)}$ to \eqref{21.07.26.15.42} and \eqref{21.07.26.15.43}, then take sum over $n\in \bZ$.  Finally,   we  use  \eqref{21.07.26.15.53} and 
 Lemma \ref{21.06.15.13.47}($ii$) to finish the proof of the theorem.
 \end{proof}

\mysection{Proof of Theorems \ref{21.07.14.13.00}, \ref{21.07.27.12.13}, \ref{21.10.18.2056} and \ref{21.10.18.2057}}
\label{subproof}

We only need to prove   Theorems \ref{21.07.14.13.00} and  \ref{21.07.27.12.13}. This is   because  Theorem \ref{21.10.18.2056} is a consequence of    Theorems \ref{21.07.14.13.00}  and 
\ref{21.06.18.1433}, and  Theorem \ref{21.10.18.2057} is a consequence of Theorems  \ref{21.07.27.12.13}  and  \ref{21.07.13.13.38}.

\vspace{2mm}

\textbf{Proof of Theorem \ref{21.07.14.13.00}}

\textbf{1}. Existence and  estimate of solution.

First, assume $u_0\in C_c^\infty(D)$ and $f\in C_c^\infty((0,T)\times D)$.  Then, by Lemma \ref{21.05.12.1839}, the function $u$ defined in \eqref{rep parabolic} becomes a weak solution to \eqref{21.06.18.1430}. Also, by   Lemmas \ref{21.06.18.1403} and \ref{21.06.18.1404}, 
\begin{align*}
    \|\psi^{-\alpha/2}u\|_{\bL_{p,\theta}(D,T)} \leq C \left(\|\psi^{-\alpha/2+\alpha/p}u_0\|_{L_{p,\theta}(D)} + \|\psi^{\alpha/2}f\|_{\bL_{p,\theta}(D,T)} \right).
\end{align*}
Now we fix $\gamma\in (0,\alpha/p)$. By \eqref{besov weight}, we have $L_{p,\theta'}(D)\subset B_{p,\theta'}^{\gamma-\alpha/p}(D)$ for any $\theta'\in\bR$, and therefore  applying Theorem \ref{21.06.18.1433} with $\mu=0$, we conclude $u\in \psi^{\alpha/2}{\bH}^{\gamma}_{p,\theta}(D,T)$ and
\begin{align*}
    \|\psi^{-\alpha /2}u\|_{\bH^{\gamma}_{p,\theta}(D,T)} \leq C \left(\|\psi^{-\alpha/2+\alpha/p}u_0 \|_{L_{p,\theta}(D)} + \|\psi^{\alpha/2} f\|_{\bL_{p,\theta}(D,T)} \right).
\end{align*}
 Using this and  Corollary \ref{cor 10.10-1}($i$), we have $u_t=\Delta^{\alpha/2}u+f\in \psi^{-\alpha/2}\bH^{\gamma-\alpha}_{p,\theta}(D,T)$, $u\in \frH^{\gamma}_{p,\theta}(D,T)$, and 
\begin{align} \label{21.10.20.1116}
    \|u\|_{\frH_{p,\theta}^\gamma(D,T)} \leq C \left(\|\psi^{-\alpha/2+\alpha/p}u_0 \|_{L_{p,\theta}(D)} + \| \psi^{\alpha/2}f\|_{\bL_{p,\theta}(D,T)} \right).
\end{align}

For general data, we take 
$\{u_{0n}\}_{n\in\bN} \subset C_c^{\infty}(D)$ and $\{f_n\}_{n\in\bN}\subset C_c^{\infty}((0,T)\times D)$ such that 
\begin{equation*} 
    \begin{gathered}
    u_{0n}\to u_0\quad \text{in}\quad \psi^{\alpha/2-\alpha/p}B_{p,\theta}^{\gamma-\alpha/p}(D),
    \\
    f_n\to f\quad \text{in}\quad \psi^{-\alpha/2}\bH_{p,\theta}^{\gamma}(D,T).
    \end{gathered}
\end{equation*}
Define $u_n$ (resp. $u$) by \eqref{rep parabolic} with $u_{0n}$ (resp. $u_0$) and $f_n$ (resp. $f$). Then, by Lemmas \ref{21.06.18.1403} and \ref{21.06.18.1404},  $u_n$ converges to $u$ in the space $\bL_{p,\theta-\alpha p/2}(D,T)$. Also, considering the estimate \eqref{21.10.20.1116} corresponding to $u_n-u_m$, we conclude $u_n$ is a Cauchy sequence in $\frH^{\gamma}_{p,\theta}(D,T)$.  Let  $v$   denote the limit of $u_n$ in  
$\frH^{\gamma}_{p,\theta}(D,T)$.  Then, $v=u$ (a.e.) and  therefore $u$ (or its version) is in $\frH^{\gamma}_{p,\theta}(D,T)$.

  Now we  prove that  \eqref{21.09.20.1816} holds for all $t\leq T$.  Since $u_n$ is a solution to \eqref{21.06.18.1430} in the sense of Definition \ref{21.06.18.1431}, taking $n\to \infty$ and using 
  $$
  \|\psi^{-\alpha/2}(u_n-u)\|_{\bH^{\gamma}_{p,\theta}(D,T)}+\|\psi^{\alpha/2}(\Delta^{\alpha/2}u_n-\Delta^{\alpha/2}u)\|_{\bH^{\gamma-\alpha}_{p,\theta}(D,T)}\to 0,
  $$
   we find that  \eqref{21.09.20.1816} holds  for $u$ almost everywhere on $[0,T]$.  By Theorem \ref{21.10.06.0929}, we know that $(u(t)-u_0,\phi)_D$ is a continuous in $t$,  and therefore we   conclude that  \eqref{21.09.20.1816} holds for all $t \leq T$. Thus, $u$ becomes a weak solution.   \eqref{para} also follows from the estimates of $u_n$.

\vspace{2mm}

\textbf{2}. Uniqueness.  

Let $u\in \psi^{\alpha/2}\bL_{p,\theta}(D,T)\cap \{u=0 \,\,\text{on}\, [0,T]\times D^c\}$ be a weak solution to
\begin{equation*}
\begin{cases}
\partial_t u(t,x)=\Delta^{\alpha/2}u(t,x),\quad &(t,x)\in(0,T)\times D,
\\
u(0,x)=0,\quad & x\in D,
\\
u(t,x)=0,\quad &(t,x)\in[0,T]\times D^c.
\end{cases}
\end{equation*}
Then, by Theorem \ref{21.06.18.1433} with $\mu=0$ and $\gamma>0$, we have $u\in \psi^{\alpha/2}\bH^{\gamma}_{p,\theta}(D,T)$ for any $\gamma>0$.  This and Corollary \ref{cor 10.10-1}($i$) imply  $u\in\frH_{p,\theta}^{\gamma+\alpha}(D,T)$ for any $\gamma \in \bR$. 

Now we take a sequence $u_n\in C_c^\infty([0,T]\times D)$  (cf. Remark  \ref{21.10.22.2123}($ii$))  such that $u_n \to u$ in $\frH_{p,\theta}^{\gamma+\alpha}(D,T)$.  In particular, $u_n \to u$, $u_n(0,\cdot)\to 0$, and $\partial_tu_n \to \partial_t u$ in their corresponding spaces.  Define  $f_n:=\partial_t u_n - \Delta^{\alpha/2}u_n$, then $u_n$ trivially satisfies
$$
\partial_t u_n-\Delta^{\alpha/2}u_n=f_n.
$$
   By Lemma \ref{21.05.12.1839}($ii$),
$$
u_n(t,x)=\bE_x[u_n(0,X_t^D)]+\int_0^t\bE_x[f_n(s,X_{t-s}^D)]ds.
$$
Also, by  Lemmas \ref{21.06.18.1403} and \ref{21.06.18.1404} and Theorem \ref{21.06.18.1433},
\begin{align}
\label{21.07.30.14.09-1}
&\|\psi^{-\alpha/2}u_n\|_{\bH_{p,\theta}^{\gamma+\alpha}(D,T)}
\\
&\leq C\left(\|\psi^{\alpha/p-\alpha/2}u_n(0,\cdot)\|_{B_{p,\theta}^{\gamma+\alpha-\alpha/p}(D)}+\|\psi^{\alpha/2}f_n\|_{\bH_{p,\theta}^{\gamma}(D,T)}\right). \nonumber
\end{align}
By Corollary \ref{cor 10.10-1}($i$), we have $\Delta^{\alpha/2}u_n \to \Delta^{\alpha/2} u$ in $\psi^{-\alpha/2}\bH^{\gamma}_{p,\theta}(D,T)$, and therefore 
$$
f_n=\partial_t u_n -\Delta^{\alpha/2}u_n \to \partial_t u-\Delta^{\alpha/2} u=0
$$
as $n\to \infty$ in the space $\psi^{-\alpha/2}\bH^{\gamma}_{p,\theta}(D,T)$.  From  \eqref{21.07.30.14.09-1},  we conclude that $u=0$.  The uniqueness is also proved.

\vspace{3mm}

\textbf{Proof of Theorem \ref{21.07.27.12.13}} 

If $\lambda>0$ or $D$ is bounded,  then it is enough to repeat the proof of Theorem \ref{21.07.14.13.00}. During the proof, one only needs to replace results for parabolic equations by their corresponding elliptic versions.

\vspace{1mm}

Therefore, we only consider the case when  $\lambda=0$ and $D$ is a half space. 

\textbf{1}. A priori estimate and uniqueness.  

We first prove the a priori estimate
\begin{align}
\label{21.10.22.0141}
    \|\psi^{-\alpha/2} u\|_{H^{\alpha}_{p,\theta}(D)}\leq C\|\psi^{\alpha/2}f\|_{L_{p,\theta}(D)}
\end{align}
holds given that  $u \in L_{p,\theta-\alpha p/2}(D)\cap \{u=0\,\text{on}\,\,D^c\}$ is a weak solution to \eqref{21.07.13.13.02}.

Note that by Theorem \ref{21.07.13.13.38}, we have $u\in \psi^{\alpha/2}H^{\alpha}_{p,\theta}(D)$. 
Assume $u\in C_c^\infty(D)$ for a moment. Then, for any $\lambda>0$,
$$
\Delta^{\alpha/2}u-\lambda u =f-\lambda u \text{ on } D,
$$
where $f:=\Delta^{\alpha/2} u$. Thus, applying \eqref{21.07.13.13.31} for $\lambda>0$ and letting $\lambda\downarrow 0$, we get estimate \eqref{21.10.22.0141} for $\lambda=0$. For general case,  we take $u_n \in C^{\infty}_c(D)$ such that $u_n \to u$ in $\psi^{\alpha/2}H^{\alpha}_{p,\theta}(D)$. Then, by Corollary  \ref{cor 10.10-1}($i$), $\Delta^{\alpha/2}u_n \to \Delta^{\alpha/2}u=f$ in $\psi^{-\alpha/2}L_{p,\theta}(D)$. Consequently, this leads to \eqref{21.10.22.0141}, which certainly implies
$$
  \|\psi^{-\alpha/2} u\|_{L_{p,\theta}(D)}\leq C\|\psi^{\alpha/2}f\|_{L_{p,\theta}(D)}.$$
   The uniqueness result of solution easily follows from this. 

\vspace{1mm}

\textbf{2}. Weak convergence and existence. 

Let $u_n \in \psi^{\alpha/2}L_{p,\theta}(D)\cap \{u=0 \,\,\text{on}\, D^c\}$ denote the solution to equation \eqref{21.07.13.13.02} corresponding to $\lambda=\frac{1}{n}$. Then, by \eqref{21.07.13.13.31}, 
$\{u_{n}\}$ is a bounded sequence in the space $L_p(\bR^d, \rho^{\theta-d-\alpha p/2} dx)$, and therefore there exists a subsequence $\{u_{n_i}\}$ which converges weakly to some $u\in  L_p(\bR^d, \rho^{\theta-d-\alpha p/2} dx)$. Obviously, we have $u=0$ (a.e.) on $D^c$.   By Lemma \ref{cor 10.10}, for any $\phi\in C^{\infty}_c(D)$, $\Delta^{\alpha/2}\phi$ belongs to the dual space of $\psi^{\alpha/2}H^{\alpha}_{p,\theta}(D)$. Therefore, 
$$
(u_{n_i}, \phi)_{\bR^d}=(u_{n_i}, \phi)_{D} \to (u,\phi)_{D}=(u,\phi)_{\bR^d}
$$
and
$$
( u_{n_i},\Delta^{\alpha/2}\phi )_{\bR^d}=(u_{n_i}, \Delta^{\alpha/2}\phi)_D  \to (u,\Delta^{\alpha/2}\phi)_D= ( u,\Delta^{\alpha/2}\phi )_{\bR^d},
$$
as $n_i\to\infty$.  Thus, we conclude $u$ is a weak solution to \eqref{21.07.13.13.02} in $L_{p,\theta-\alpha p/2}(D)\cap \{u=0\, \text{on}\, D^c\}$.  Now we prove the  weak convergence. The above argument shows that   any subsequence of $u_n$ has a further subsequence which converges  weakly  in $L_p(\bR^d, \rho^{\theta-d-\alpha p/2} dx)$, and the limit becomes a solution to \eqref{21.07.13.13.02} in $L_{p,\theta-\alpha p/2}(D)\cap \{u=0\, \text{on}\, D^c\}$. Due to the uniqueness of solution proved above, we conclude that this limit coincides with $u$. This proves the weak convergence, and 
the theorem is proved.

\appendix

\mysection{Auxiliary results} \label{21.09.25.1532}

Recall that $p(t,x)=p_d(t,x)$ is the transition density function of a rotationally symmetric $\alpha$-stable $d$-dimensional L\'evy process. It  is a radial function and 
\begin{align}
\label{21.06.22.15.41}
p_d(t,x) &\approx t^{-\frac{d}{\alpha}}\wedge \frac{t}{|x|^{d+\alpha}} \approx \frac{t}{(t^{1/\alpha}+|x|)^{d+\alpha}}, 
\end{align}
and
\begin{equation} \label{21.06.22.15.40}
p_d(t,x)=t^{-\frac{d}{\alpha}}p_d(1,t^{-\frac{1}{\alpha}}x).
\end{equation}

If $f$ is a radial function, then we put $f(r):=f(x)$ if $r=|x|$.
\begin{lemma}
(i) Let $d\geq 2$ and $f$ be a nonnegative radial function  on $\bR^d$. Then, for $x^1\neq 0$, 
\begin{align} \label{21.06.15.1756}
    \int_{\bR^{d-1}} f(x^1,x')dx' = C(d)|x^1|^{d-1}\int_0^\infty f(|x^1|(1+s^2)^{1/2})s^{d-2}ds.
\end{align}

(ii) Let $d\geq 2$. For any $t>0$ and $x^1\neq 0$, 
\begin{align}
\label{21.08.02.16.12}
    \int_{\bR^{d-1}} p_d(t,x^1,x')dx'\approx p_1(t,x^1),
\end{align}
where the comparability relation depends only on $d$ and $\alpha$.
\end{lemma}
\begin{proof}
($i$) By the change of variables,
\begin{align*}
\int_{\bR^{d-1}} f(x^1,x')dx' &= \int_{\bR^{d-1}} f(x^1,|x^1|x')|x^1|^{d-1}dx'
\\
&=|x^1|^{d-1}\int_{\bR^{d-1}} f(|x^1|(1+|x'|^2)^{1/2})dx'
\\
&=C(d) |x^1|^{d-1}\int_0^\infty f(|x^1|(1+s^2)^{1/2})s^{d-2} ds.
\end{align*}

($ii$)
 By \eqref{21.06.22.15.41} and \eqref{21.06.15.1756}, it suffices to prove that
\begin{align}
\label{21.03.22.14.59}
    \int_0^\infty \frac{t|x^1|^{d-1}s^{d-2} }{(t^{1/\alpha}+|x^1|(1+s^2)^{1/2})^{d+\alpha}} ds \approx \left(t^{-\frac{1}{\alpha}}\wedge\frac{t}{|x^1|^{1+\alpha}}\right).
\end{align}
Let $t|x^1|^{-\alpha}\leq1$, then
\begin{align} \nonumber
    \int_0^\infty \frac{t|x^1|^{d-1}s^{d-2} }{(t^{1/\alpha}+|x^1|(1+s^2)^{1/2})^{d+\alpha}} ds&\approx  \int_0^\infty \frac{t|x^1|^{d-1}s^{d-2}}{\left(|x^1|(1+s^2)^{1/2}\right)^{d+\alpha}} ds
    \\
    &= C(d,\alpha) \frac{t}{|x^1|^{1+\alpha}}.  \label{equiv}
\end{align}
 Now let  $t|x^1|^{-\alpha}\geq1$. We put  
\begin{align*} \label{21.06.16.1806}
    \int_0^\infty \frac{t|x^1|^{d-1}s^{d-2} }{(t^{1/\alpha}+|x^1|(1+s^2)^{1/2})^{d+\alpha}} ds = \int_0^{t^{1/\alpha}|x^1|^{-1}}\cdots +\int_{t^{1/\alpha}|x^1|^{-1}}^{\infty}\cdots  =:I+II.
\end{align*}
Then,
\begin{equation*}
    \begin{gathered}
    I\leq t^{-\frac{d}{\alpha}}|x^1|^{d-1}\int_0^{t^{1/\alpha}|x^1|^{-1}}s^{d-2}ds=C(d,\alpha)t^{-\frac{1}{\alpha}},
    \\
    II\leq t|x^1|^{-1-\alpha}\int_{t^{1/\alpha}|x^1|^{-1}}^{\infty} \frac{s^{d-2}}{(1+s^2)^{\frac{d+\alpha}{2}}} ds = C(\alpha)t^{-\frac{1}{\alpha}}.
    \end{gathered}
\end{equation*}
Therefore, the left-hand side of \eqref{21.03.22.14.59} is controlled by the right-hand side.  Due to \eqref{equiv}, to prove \eqref{21.03.22.14.59}, we only need a proper  lower bound of $I$.
 Let $t|x^1|^{-\alpha}\geq1$.  By the changing variables $s=t^{1/\alpha}|x^1|^{-1}l$,
\begin{align*}
    I &\geq\int_0^{t^{1/\alpha}|x^1|^{-1}} \frac{t|x^1|^{d-1}s^{d-2} }{(t^{1/\alpha}+(t^{2/\alpha}+|x^1|^{-2}s^2)^{1/2})^{d+\alpha}}ds
    \\
    &=t^{-\frac{1}{\alpha}}\int_0^1\frac{l^{d-2}}{(1+(1+l^2)^{1/2})^{d+\alpha}}dl=C(d,\alpha)t^{-\frac{1}{\alpha}}.
\end{align*}
The lemma is proved.
\end{proof}

\begin{lemma}
\label{re1}
Let $\alpha\in(0,2)$ and $\gamma_0$, $\gamma_1\in\bR$.
Suppose that
\begin{equation}
\label{21.05.13.13.02}
-\frac{2}{\alpha}<\gamma_0,\quad -2<\gamma_1-\gamma_0\leq2+\frac{2}{\alpha}.
\end{equation}
 Then, for any $(t,x)\in (0,\infty)\times\bR^d $, 
    \begin{align} \label{21.06.15.2038}
    \int_{\bR^d}p(t,x-y)\frac{|y^1|^{\gamma_0\alpha/2}}{(\sqrt{t}+|y^1|^{\alpha/2})^{\gamma_1}}dy\leq C(\sqrt{t}+|x^1|^{\alpha/2})^{\gamma_0-\gamma_1}.
    \end{align}
    where $C=C(d,\alpha,\gamma_0,\gamma_1)$. 
\end{lemma}

\begin{proof}
 It suffices to prove \eqref{21.06.15.2038}  when $t=1$.  Indeed, if it holds for $t=1$, then  by \eqref{21.06.22.15.40},
\begin{align*}
&\int_{\bR^d} p(t,x-y) \frac{|y^1|^{\gamma_0\alpha/2}}{(\sqrt{t}+|y^1|^{\alpha/2})^{\gamma_1}} dy
\\
&= C t^{\frac{\gamma_0-\gamma_1}{2}}\int_{\bR^d}p(1,t^{-\frac{1}{\alpha}}x-y)\frac{|y^1|^{\gamma_0\alpha/2}}{(1+|y^1|^{\alpha/2})^{\gamma_1}}dy
\\
&\leq C t^{\frac{\gamma_0-\gamma_1}{2}}(1+t^{-\frac{1}{2}}|x^1|^{\alpha/2})^{\gamma_0-\gamma_1}
\\
&=C (\sqrt{t}+|x^1|^{\alpha/2})^{\gamma_0-\gamma_1}.
\end{align*}

Thus, we may assume $t=1$. By \eqref{21.08.02.16.12} and \eqref{21.06.22.15.41},
\begin{align*}
   & \int_{\bR^d}p_d(1,x-y)\frac{|y^1|^{\gamma_0\alpha/2}}{(1+|y^1|^{\alpha/2})^{\gamma_1}}dy
   \\
    &\approx \int_{\bR}\left(1\wedge\frac{1}{|x^1-y^1|^{1+\alpha}}\right)\frac{|y^1|^{\gamma_0\alpha/2}}{(1+|y^1|^{\alpha/2})^{\gamma_1}}dy^1
    =:I(x^1).
\end{align*}
Thus, it only remains to show for $x^1\in\bR$,
\begin{align}
\label{21.08.02.17.25}
    I(x^1)\leq C(1+|x^1|^{\alpha/2})^{\gamma_0-\gamma_1}.
\end{align}

Case 1.  Let $|x^1|\leq 1$.   Put
\begin{align*} 
I(x^1)&= \int_{|y^1|\leq 2}\cdots dy^1+\int_{|y^1|>2}\cdots dy^1=:I_1(x^1)+I_2(x^1).
\end{align*}
If $|y^1|\leq2$, then by \eqref{21.05.13.13.02}, 
\begin{align*} 
    I_1(x^1)\leq C \int_{|y^1|\leq 2}|y^1|^{\alpha\gamma_0/2}dy^1=C.
\end{align*}
If $|y^1|>2$, then $|x^1-y^1|\geq |y^1|/2$. Thus, by \eqref{21.05.13.13.02},
\begin{eqnarray*} 
    I_2(x^1)&\leq& C\int_{|y^1|>2}\frac{1}{|x^1-y^1|^{1+\alpha}}\left(\frac{|y^1|^{\alpha/2}}{1+|y^1|^{\alpha/2}}\right)^{\gamma_1}|y^1|^{\alpha(\gamma_0-\gamma_1)/2}dy^1 
    \\
    &\leq& C\int_{|y^1|>2}|y^1|^{\frac{\alpha(\gamma_0-\gamma_1-2)}{2}-1}dy^1=C.
\end{eqnarray*}
Therefore,  $I$ is bounded and \eqref{21.08.02.17.25} is proved for $|x^1|\leq 1$. 

\vspace{1mm}

Case 2. Let $|x^1|>1$. Put
\begin{align*}
    I(x^1)&= \int_{|y^1|\geq2|x^1|} \cdots +\int_{|x^1|/2<|y^1|<2|x^1|} \cdots +\int_{1/2<|y^1|\leq|x^1|/2} \cdots +\int_{|y^1|\leq1/2} \cdots \nonumber
    \\
    &=: J_1(x^1)+J_2(x^1)+J_3(x^1)+J_4(x^1).  
    \end{align*}
First, we estimate $J_1$. Note that if $r>1$, then
\begin{align} \label{21.06.15.1347}
\frac{1}{2}\leq \frac{r^{\alpha/2}}{1+r^{\alpha/2}}\leq 1.
\end{align}
For $|y^1|>2|x^1|$, we have $|y^1|>2$.  Thus, \eqref{21.05.13.13.02} and \eqref{21.06.15.1347} yield
\begin{align} \label{21.06.16.1528}
J_1(x^1) &\leq \int_{|y^1|\geq 2|x^1|}\frac{1}{|x^1-y^1|^{1+\alpha}}\left(\frac{|y^1|^{\alpha/2}}{1+|y^1|^{\alpha/2}}\right)^{\gamma_1}|y^1|^{\frac{\alpha(\gamma_0-\gamma_1)}{2}} dy^1 \nonumber
\\
&\leq C\int_{|y^1|\geq 2|x^1|}|y^1|^{\frac{\alpha(\gamma_0-\gamma_1-2)}{2}-1}dy^1 \nonumber
\\
&= C|x^1|^{\frac{\alpha(\gamma_0-\gamma_1-2)}{2}} \leq C|x^1|^{\frac{\alpha(\gamma_0-\gamma_1)}{2}} \leq C(1+|x^1|^{\alpha/2})^{\gamma_0-\gamma_1}.
\end{align}
Secondly, we estimate $J_2$. If $|x^1|/2<|y^1|<2|x^1|$, then
$$
\frac{1}{2} \leq \frac{1+|y^1|^{\alpha/2}}{1+|x^1|^{\alpha/2}} \leq 2 ,\quad \frac{1}{3}\leq\frac{|y^1|^{\alpha/2}}{1+|y^1|^{\alpha/2}}\leq1.
$$
Therefore, we have
\begin{align} \label{21.06.16.1529}
J_2(x^1)&\leq C(1+|x^1|^{\alpha/2})^{\gamma_0-\gamma_1}\int_{\bR} p_1(1,x^1-y^1) dy^1 \nonumber
\\
&= C(1+|x^1|^{\alpha/2})^{\gamma_0-\gamma_1}.
\end{align}
Next, we estimate $J_3$. If $1/2 \leq |y^1| \leq |x^1|/2$, then
$$
\frac{1}{3}\leq\frac{|y^1|^{\alpha/2}}{1+|y^1|^{\alpha/2}}\leq1,\quad |x^1-y^1|\geq \frac{|x^1|}{2}.
$$
Hence, by \eqref{21.05.13.13.02} and \eqref{21.06.15.1347},
\begin{align} \label{21.06.16.1530}
    J_3(x^1) &\leq \int_{1/2 \leq |y^1|\leq|x^1|/2} \frac{1}{|x^1-y^1|^{1+\alpha}}\left(\frac{|y^1|^{\alpha/2}}{1+|y^1|^{\alpha/2}}\right)^{\gamma_1}|y^1|^{\frac{\alpha(\gamma_0-\gamma_1)}{2}} dy^1 \nonumber
    \\
    &\leq C|x^1|^{-1-\alpha}\int_{1/2 \leq |y^1| \leq |x^1|/2} |y^1|^{\frac{\alpha(\gamma_0-\gamma_1)}{2}} dy^1 \nonumber
    \\
    &\leq C|x^1|^{-1-\alpha}\int_{1/2 \leq |y^1| \leq |x^1|/2} |y^1|^{\frac{\alpha(\gamma_0-\gamma_1+2)}{2}} dy^1 \nonumber
    \\
    &\leq C|x^1|^{-1-\alpha}\int_{|y^1|\leq |x^1|/2}|y^1|^{\frac{\alpha(\gamma_0-\gamma_1+2)}{2}}dy^1 \nonumber
    \\
    &= C|x^1|^{\frac{\alpha(\gamma_0-\gamma_1)}{2}}\leq C (1+|x^1|^{\alpha/2})^{\gamma_0-\gamma_1}.
\end{align}
Lastly, we estimate $J_4$. If $|y^1|\leq 1/2$, then
$$
\frac{2}{3}\leq\frac{1}{1+|y^1|^{\alpha/2}}\leq 1,\quad |x^1-y^1|\geq \frac{|x^1|}{2}.
$$
Therefore, by \eqref{21.05.13.13.02} and \eqref{21.06.15.1347},
\begin{align*}
    J_4(x^1)&\leq \int_{|y^1|\leq 1/2} \frac{1}{|x^1-y^1|^{1+\alpha}}\left(\frac{1}{1+|y^1|^{\alpha/2}}\right)^{\gamma_1}|y^1|^{\alpha\gamma_0/2}dy^1
    \\
    &\leq C|x^1|^{-1-\alpha}\int_{|y^1|\leq 1}|y^1|^{\alpha\gamma_0/2}dy^1 
    \\
    &\leq C|x^1|^{-1-\alpha}\leq C(1+|x^1|^{\alpha/2})^{-2/\alpha-2} \leq C(1+|x^1|^{\alpha/2})^{\gamma_0-\gamma_1}.
\end{align*}
Combining this with \eqref{21.06.16.1528}, \eqref{21.06.16.1529} and \eqref{21.06.16.1530},  we prove \eqref{21.08.02.17.25}  for 
$|x^1| > 1$. The lemma is proved.
\end{proof}

\begin{lemma}
\label{re2}
 Let  \eqref{21.05.13.13.02} hold for $\gamma_0, \gamma_1\in \bR$. Then, for $(t,x)\in (0,\infty)\times\bR^d$,
    \begin{align*}
        \int_{D}p(t,x-y)\frac{d_y^{\gamma_0\alpha/2}}{(\sqrt{t}+d_y^{\alpha/2})^{\gamma_1}}dy\leq C (\sqrt{t}+d_x^{\alpha/2})^{\gamma_0-\gamma_1},
    \end{align*}
where $C$ depends only on $d,\alpha,\gamma_0,\gamma_1$ and $D$.
\end{lemma}

\begin{proof}
Note that it is enough to assume $D$ is bounded. This is because if $D$ is a half space, the result follows from Lemma \ref{re1}.

For $R>0$, denote  $D_R:=\{x\in D: d_x\geq R\}$. Since $D$ is bounded, one can find $x_1,\dots, x_n\in \partial D$ such that 
$$
D \subset \left(\bigcup_{i=1}^n (D\cap B_{R/3}(x_i))\right) \cup D_{R/6}.
$$
Therefore,
\begin{align*} 
&\int_{D} p(t,x-y) \frac{d_y^{\alpha\gamma_0/2}}{(\sqrt{t}+d_y^{\alpha/2})^{\gamma_1}} dy
\\
&\leq \sum_{i=1}^n \int_{D\cap B_{R/3}(x_i)} p(t,x-y) \frac{d_y^{\alpha\gamma_0/2}}{(\sqrt{t}+d_y^{\alpha/2})^{\gamma_1}} dy 
\\
&\quad + \int_{D_{R/6}} p(t,x-y) \frac{d_y^{\alpha\gamma_0/2}}{(\sqrt{t}+d_y^{\alpha/2})^{\gamma_1}} dy
\\
&=:\sum_{k=1}^n I_k(t,x)+II(t,x).
\end{align*}

\textbf{1}.  We estimate $I_k(t,x)$ for fixed  $k\in \{1,2,\cdots,n\}$.

 First, assume $x\in B_R(x_k)\cap D$. Then, (by reducing $R$ if necessary) we can consider a $C^{1,1}$-bijective (flattening boundary) map $\Phi=(\Phi^1,\cdots,\Phi^d)$ defined on $B_R(x_k)$ such that $\Phi(B_R(x_k)\cap D)\subset \bR^d_+$ and $d_z \approx \Phi^1(z)$ on $B_R(x_k)\cap D$.  Then, one can easily handle $I_k$ using Lemma \ref{re1}.

Second, assume $x\in D\setminus B_R(x_k)$.  Since $r\to p(t,r)$ is nonincreasing, for any $y,z \in B_{R/3}(x_k)$, we have $|z-y| \leq 2R/3 < |x-y|$, which implies
\begin{align*} 
p(t,x-y)\leq p(t,z-y).
\end{align*}
If $\gamma_1-\gamma_0\geq 0$, choosing $z \in B_{R/3}(x_k) \cap D$ such that $d_x \leq C(D,R) d_z$ and using the result for the first case, 
\begin{align} \label{21.05.13.17.01}
I_k(t,x) &\leq \int_{D\cap B_{R/3}(x_k)} p(t,z-y) \frac{d_y^{\alpha\gamma_0/2}}{(\sqrt{t}+d_y^{\alpha/2})^{\gamma_1}} dy \nonumber
\\
& \leq  C (\sqrt{t}+d_z^{\alpha/2})^{\gamma_0-\gamma_1} \leq C (\sqrt{t}+d_x^{\alpha/2})^{\gamma_0-\gamma_1}.
\end{align}    
If $\gamma_1-\gamma_0 <0$, by taking $z \in B_{R/3}(x_k)\cap D$ such that $d_z\leq d_x$, we also have \eqref{21.05.13.17.01}.

\vspace{2mm}

\textbf{2}.  We estimate $II(t,x)$.

We first consider the case  $x\in D_{R/12}$. For $y\in D_{R/6}$,   we have $d_x\approx d_y \approx 1$ and 
\begin{equation*} 
\left(\frac{\sqrt{t}+d_y^{\alpha/2}}{\sqrt{t}+d_x^{\alpha/2}}\right)^{\gamma_0-\gamma_1} \leq C(diam(D),\gamma_0,\gamma_1,R,\alpha),
\end{equation*}
Using this, we get
\begin{equation*}
    II \leq C (\sqrt{t}+d_x^{\alpha/2})^{\gamma_0-\gamma_1}\int_{D_{R/6}}p(t,x-y)\left(\frac{d_y^{\alpha/2}}{\sqrt{t}+d_y^{\alpha/2}}\right)^{\gamma_0} dy.
\end{equation*}
Also, since $d_y \approx 1$ on $y\in D_{R/6}$, it suffices to show that 
\begin{align} \label{21.06.16.1754}
    \int_{D_{R/6}} p(t,x-y) \left(\frac{1}{\sqrt{t}+1}\right)^{\gamma_0} dy \leq C.
\end{align}
Since \eqref{21.06.16.1754} is obvious if $t\leq 1$ or $\gamma_0\geq 0$.  If $t>1$ and $\gamma_0<0$ , then 
 by \eqref{21.06.22.15.41},
\begin{align*}
\int_{D_{R/6}} p(t,x-y)\left(\frac{1}{\sqrt{t}+1}\right)^{\gamma_0} dy &\leq C  \int_{D}  t^{-d/\alpha-\gamma_0/2} dy \leq C.
\end{align*}
Therefore, \eqref{21.06.16.1754} is proved.

Next, we consider the case $x\in D\setminus D_{R/12}$. Since $d_y\approx 1$, we have
$$
\frac{d_y^{\alpha\gamma_0/2}}{(\sqrt{t}+d_y^{\alpha/2})^{\gamma_1}} \approx  \frac{1}{(\sqrt{t}+1)^{\gamma_1}}.
$$
Also note that $|x-y|>R/12$ for $y\in D_{R/6}$. 
Thus, by \eqref{21.06.22.15.41},
 \begin{eqnarray*} 
    II &\leq& C 1_{t<1}\int_{|x-y|\geq R/12} \frac{t}{|x-y|^{d+\alpha}} dy + C1_{t\geq 1}  t^{-d/\alpha-\gamma_1/2} 
    \\
    &\leq& C  1_{t<1} + C 1_{t\geq 1}t^{-\gamma_0/2-d/\alpha} t^{(\gamma_0-\gamma_1)/2}  \\
    &\leq& C  1_{t<1} + C 1_{t\geq 1} t^{(\gamma_0-\gamma_1)/2}. 
    \end{eqnarray*}
Thus if  $\gamma_0\geq \gamma_1$, then by \eqref{21.05.13.13.02},
$$
    II\leq C t^{(\gamma_0-\gamma_1)/2} \leq C (\sqrt{t}+d_x^{\alpha/2})^{\gamma_0-\gamma_1}.
$$
Now let $\gamma_0<\gamma_1$. Then, $1_{t<1} (\sqrt{t}+d^{\alpha/2}_x)$ is bounded above and $t \approx (t+d^{\alpha/2}_x)$ if $t>1$, we get
$$
II \leq C 1_{t<1} + C 1_{t\geq 1} t^{(\gamma_0-\gamma_1)/2} \leq C (\sqrt{t}+d_x^{\alpha/2})^{\gamma_0-\gamma_1}
$$
provided that $\gamma_0<\gamma_1$. The lemma is proved.
\end{proof}

Next, we provide some results for the distance function $d_x$.

\begin{lemma}
\label{21.05.13.11.18}
Let $D$ be a half space or a bounded $C^{1,1}$ open set.
\begin{enumerate}[(i)]
    \item Let $x_0\in\partial D$ and $r>0$. Then, for any $\lambda>-1$, 
\begin{align} \label{21.09.20.1850}
\aint_{B_r(x_0)}d_x^{\lambda}dx\leq C(d,\lambda,D) r^{\lambda}.
\end{align}

    \item Let $y\in D$, $r, \rho, \kappa_1>0$ and $-1<\kappa_0\leq0$. Suppose that $r\leq c\rho$ for some $c>0$. Then, there exists a constant $C=C(d,\kappa_1,\kappa_0,c,D)$ such that 
$$
\int_{ D_{\rho}(y)\cap D^r}\frac{d_x^{\kappa_0}}{|x-y|^{d+\kappa_1}}dx \leq C \rho^{-\kappa_1}r^{\kappa_0},
$$
where $D_\rho(y):=\{x\in D:|x-y|>\rho\}$ and $D^r:=\{x\in D:d_x \leq r\}$.

\end{enumerate}
\end{lemma}

\begin{proof}
($i$)  The result is trivial if $D$ is a half space. If $D$ is a bounded $C^{1,1}$ open set, then $\partial D$ is a $(d-1)$-dimensional compact Lipschitz manifold.  Thus, we have \eqref{21.09.20.1850}  due to  e.g. page 16 of \cite{aikawa1991quasiadditivity}. 

$(ii)$ \textbf{1}. Let $D$ be a half space.

Assume first $d\geq2$.
By  the change of variables  and Fubini's theorem,
\begin{align} \label{21.06.07.1538}
&\int_{|x-y|>\rho,|x^1|\leq r}\frac{|x^1|^{\kappa_0}}{|x-y|^{d+\kappa_1}}dx \nonumber\\
&= \int_{|x^1+y^1|\leq r} |x^1+y^1|^{\kappa_0} \int_{\bR^{d-1}} |x|^{-d-\kappa_1}1_{|x|>\rho} dx' dx^1 \nonumber
\\
&=C\int_{|x^1+y^1|\leq r} \frac{|x^1+y^1|^{\kappa_0}}{|x^1|^{1+\kappa_1}}\int_0^{\infty}\frac{s^{d-2}}{(1+s^2)^{(d+\kappa_1)/2}}1_{|x^1|(1+s^2)^{1/2}>\rho}dsdx^1 \nonumber
\\
&=C\int_0^{\infty}\frac{s^{d-2}}{(1+s^2)^{(d+\kappa_1)/2}}I(\rho,s,y^1,r)ds,
\end{align}
where
$$
I(\rho,s,y^1,r):=\int_{\bR}\frac{|x^1|^{\kappa_0}}{|x^1-y^1|^{1+\kappa_1}}1_{|x^1-y^1|>(1+s^2)^{-1/2}\rho}1_{|x^1|\leq r}dx^1.
$$
Take $p_0=p_0(\kappa_0)>1$ satisfying $-1<p_0\kappa_0$. Since $-1<\kappa_0\leq0<\kappa_1$, by H\"older's inequality,
\begin{align}
\label{21.06.29.13.55}
    &I(\rho,s,y^1,r)\nonumber\\
    &\leq\left(\int_{\bR}|x^1|^{p_0\kappa_0}1_{|x^1|\leq r}dx^1\right)^{1/p_0}\left(\int_{\bR}|x^1|^{-p_0'-p_0'\kappa_1}1_{|x^1|>(1+s^2)^{-1/2}\rho}dx^1\right)^{1/p_0'}\nonumber
    \\ 
    &\leq Cr^{\kappa_0+\frac{1}{p_0}}\rho^{-1-\kappa_1+\frac{1}{p_0'}}(1+s^2)^{\frac{(1+\kappa_1-1/p_0')}{2}},
\end{align}
where $p_0'=p_0/(p_0-1)$.
Combining \eqref{21.06.07.1538} and \eqref{21.06.29.13.55}, we have
\begin{align*}
    &\int_{|x-y|>\rho,|x^1|\leq r}\frac{|x^1|^{\kappa_0}}{|x-y|^{d+\kappa_1}}dx
    \\
    &\leq C\rho^{-1-\kappa_1+\frac{1}{p_0'}}r^{\kappa_0+\frac{1}{p_0}}\int_0^{\infty}\frac{s^{d-2}}{(1+s^2)^{(d-1+1/p_0')/2}}ds
    \\
    &=C\rho^{-1-\kappa_1+\frac{1}{p_0'}}r^{\kappa_0+\frac{1}{p_0}}\leq C \rho^{-\kappa_1}r^{\kappa_0}.
\end{align*}

 For $d=1$, using \eqref{21.06.29.13.55}, we get
 \begin{align*}
 \int_{|x-y|>\rho,|x|\leq r}\frac{|x|^{\kappa_0}}{|x-y|^{1+\kappa_1}}dx &= I(\rho,0,y,r) 
 \\
 &\leq C r^{\kappa_0+\frac{1}{p_0}}\rho^{-1-\kappa_1+\frac{1}{p_0'}} \leq C \rho^{-\kappa_1}r^{\kappa_0}.
 \end{align*}

\textbf{2}. Let $D$ be a bounded open set.
We take $x_1,\dots, x_n\in \partial D$ such that
$$
D^r \subset \bigcup_{i=1}^{n}B_{2r}(x_i).
$$
Therefore, by ($i$),
\begin{align*}
\int_{D_{\rho}(y)\cap D^r} \frac{d_x^{\kappa_0}}{|x-y|^{d+\kappa_1}} dx &\leq \sum_{i=1}^n\int_{D_{\rho}(y)\cap B_{2r}(x_i)} \frac{d_x^{\kappa_0}}{|x-y|^{d+\kappa_1}} dx \nonumber
\\
    &\leq C\rho^{-d-\kappa_1}r^{d+\kappa_0}\leq C\rho^{-\kappa_1}r^{\kappa_0}.
\end{align*}
The lemma is proved.
\end{proof}

  We write $u\in \mathcal{H}^{\gamma+\alpha}_{p}(T)$ if $u\in \bH^{\gamma+\alpha}_p(T)$, $u(0,\cdot)\in B_{p}^{\gamma+\alpha-\alpha/p}$ and there exists $f\in \bH^{\gamma}_p(T)$ such that for any 
  $\phi\in C^{\infty}_c(\bR^d)$,
$$
(u(t,\cdot), \phi)_{\bR^d}=(u(0,\cdot),\phi)_{\bR^d} +\int^t_0 (f(s,\cdot), \phi)_{\bR^d}ds, \quad \forall \, t\leq T.
$$
In this case, we write $f=u_t$. The norm in  $ \cH_{p}^{\gamma+\alpha}(T)$ is defined as 
\begin{align*}
\|u\|_{\cH_{p}^{\gamma+\alpha}(T)} := \| u\|_{\bH_{p}^{\gamma+\alpha}(T)} + \|u_t\|_{\bH_{p}^{\gamma}(T)}+ \| u(0,\cdot) \|_{B_{p}^{\gamma+\alpha-\alpha /p}}. \nonumber
\end{align*}

\begin{lemma}\label{21.10.06.15.11}
Let $p\in(1,\infty)$, $\alpha\in(0,2)$, $\gamma\in \bR$ and $1/p<\nu\leq1$. For $a>0$, $0\leq s\leq t\leq T$ and $u\in \cH^{\gamma+\alpha}_{p}(T)$, 
\begin{align}
\label{21.07.08.11.00}
    &\|u(t)-u(s)\|_{H_p^{\gamma+\alpha-\nu\alpha}}\nonumber
    \\
    &\leq C|t-s|^{\nu-1/p}a^{2\nu-1}\left(a\|u\|_{\bH_p^{\gamma+\alpha}(T)}+a^{-1}\| u_t\|_{\bH_p^{\gamma}(T)}\right),
\end{align}
where $C=C(\alpha,p,\nu)$. In particular, $C$ is independent of $T$ and $a$.
\end{lemma}
\begin{proof}
One can prove the lemma by following the proof of \cite[Theorem 7.3]{krylov2001some}, which treats the case $\alpha=2$.  
First, we note that 
due to the isometry  $(1-\Delta)^{\sigma/2}: H_p^{\gamma} \to H^{\gamma-\sigma}_p$, we only need to prove for any particular $\gamma\in \bR$, and therefore we 
assume  $\gamma=\nu\alpha-\alpha$. Second, since $C^{\infty}_c([0,T]\times \bR^d)$ is dense in $\cH^{\gamma+\alpha}_{p}(T)$, we may further assume $u\in C^{\infty}_c([0,T]\times \bR^d)$.  Third, due to the scaling argument used at the beginning of the proof of  \cite[Theorem 7.3]{krylov2001some},  it is enough to consider the case $a=T=1$. 

Finally,  to prove \eqref{21.07.08.11.00} for the case $a=T=1$,  we just need to repeat the proof  of \cite[Theorem 7.2]{kry99analytic} word for word.  Although \cite[Theorem 7.2]{kry99analytic} handles the case $\alpha=2$,  its proof works also for $\alpha\in (0,2)$ thanks to  \cite[Lemma A.2]{han2021regularity}. The lemma is proved.
\end{proof}


\begin{thebibliography}{10}

\bibitem{Abdel Global}
B. Abdellaoui, A. J. Fern\'andez, T. Leonori, A. Younes, Global fractional Calder\'on-Zygmund regularity, preprint, 2021, arXiv:2107.06535.


\bibitem{aikawa1991quasiadditivity}
H. Aikawa, Quasiadditivity of Riesz capacity, \textit{Math. Scand.} \textbf{69} (1991), no.1, 15-30.

\bibitem{arapostathis2016dirichlet}
A. Arapostathis, A. Biswas, L. Caffarelli, The Dirichlet problem for stable-like operators and related probabilistic representations, \textit{Comm. Partial Differential Equations} \textbf{41} (2016), no.9, 1472-1511.


\bibitem{bae2015schauder}
J. Bae, M. Kassmann, Schauder estimates in generalized H\" older spaces, preprint, 2015, arXiv:1505.05498.

\bibitem{baeumer2018space}
B. Baeumer, T. Luks, M.M. Meerschaert, Space‐time fractional Dirichlet problems. \textit{Math. Nachr.} \textbf{291} (2018), no.17-18, 2516-2535.

\bibitem{bass2009regularity}
R.F. Bass, Regularity results for stable-like operators, \textit{J. Funct. Anal.} \textbf{257} (2009), no.8, 2693-2722.

\bibitem{zuazua2017}
U. Biccari, M. Warma, E. Zuazua, Local elliptic regularity for the Dirichlet fractional Laplacian. \textit{Adv. Nonlinear Stud.} \textbf{17} (2017), no.2, 387-409.


\bibitem{biccari2018local}
U. Biccari, M. Warma, E. Zuazua, Local Regularity for fractional heat
equations in \textit{Recent advances in PDEs: analysis, numerics and control}, SEMA SIMAI Springer Ser., Vol.17 (2018), Springer, Cham, 233–249.

\bibitem{bogdan2014dirichlet}
K. Bogdan, T. Grzywny, M. Ryznar, Dirichlet heat kernel for unimodal Lévy processes. \textit{Stochastic Process. Appl.} \textbf{124} (2014), no.11, 3612-3650.

\bibitem{bogdan2020extension}
K. Bogdan, T. Grzywny, K. Pietruska-Pa{\l}uba, A. Rutkowski,
Extension and trace for nonlocal operators, \textit{J. Math. Pures Appl.} \textbf{137} (2020), 33-69.


\bibitem{bottcher2013levy}
B. B\"ottcher, R.L. Schilling, J. Wang, \textit{L\'evy-type processes: construction, approximation and sample path properties}, Lecture Notes in Mathematics Vol. 2099 (vol. III of the “L\'evy Matters” subseries), Springer, 2013.


\bibitem{caffarelli2009regularity}
L. Caffarelli, L. Silvestre, Regularity theory for fully nonlinear integro‐differential equations, \textit{Comm. Pure Appl. Math.}, \textbf{62} (2009), no.5, 597-638.


\bibitem{chen2010heat}
Z.Q. Chen, P. Kim, R. Song, Heat kernel estimates for the Dirichlet fractional Laplacian, \textit{J. Eur. Math. Soc.} \textbf{12} (2010), no.5, 1307-1329.

\bibitem{chung1986doubly}
K.L. Chung, Doubly-Feller process with multiplicative functional in \textit{Seminar on stochastic processes, 1985}, Progr. Probab. Statist. Vol. 12, Birkh\"auser Boston, 1986, 63-78.

\bibitem{chung2012brownian}
K.L. Chung,  Z. Zhao, \textit{From Brownian motion to Schr\"odinger’s equation}, A Series of Comprehensive Studies in Mathematics Vol. 312, Springer, 1995.

\bibitem{cozzi2017interior}
M. Cozzi, Interior regularity of solutions of non-local equations in Sobolev and Nikol’skii spaces, \textit{Ann. Mat. Pura Appl.} \textbf{196} (2017), no.2, 555-578.

\bibitem{dong2012lp}
H. Dong,  D. Kim, On $L_p$-estimates for a class of non-local elliptic equations, \textit{J. Funct. Anal.} \textbf{262} (2012), no.3, 1166-1199.

\bibitem{dong2013schauder}
H. Dong,  D. Kim, Schauder estimates for a class of non-local elliptic equations, \textit{Discrete Contin. Dyn. Syst.} \textbf{33} (2014), no.6, 2319-2347.


\bibitem{dyda2019muckenhoupt}
B. Dyda, L. Ihnatsyeva, J. Lehrb\"ack, H. Tuominen, A.V. V\"ah\"akangas, Muckenhoupt $A_p$-properties of Distance Functions and Applications to Hardy–Sobolev-type Inequalities. \textit{Potential Anal.} \textbf{50} (2019), no.1, 83-105.


\bibitem{felsinger2015dirichlet}
M. Felsinger, M. Kassmann, P. Voigt, The Dirichlet problem for nonlocal operators, \textit{Math. Z.} \textbf{279} (2015), no.3-4, 779-809.

\bibitem{fernandez2017regularity}
X. Fern\'andez-Real,  X. Ros-Oton, Regularity theory for general stable operators: parabolic equations, \textit{J. Funct. Anal.} \textbf{272} (2017), no.10, 4165-4221.


\bibitem{grafakos2014classical}
L. Grafakos, \textit{Classical Fourier Analysis}, 3rd ed., Graduate Texts in Mathematics Vol. 249, Springer, New York, 2014.


\bibitem{GT}
D. Gilbarg, N. S. Trudinger, \textit{Elliptic Partial Differential Equations of Second Order},
Reprint of the 1998 edition, Classics in Mathematics, Springer-Verlag, Berlin, 2001.

\bibitem{grubb2014local}
G. Grubb, Local and nonlocal boundary conditions for $\mu$-transmission and fractional elliptic pseudodifferential operators, \textit{Anal. PDE} \textbf{7} (2014), no.7, 1649–1682.

\bibitem{grubb2018regularity}
G. Grubb, Regularity in $L_p$ Sobolev spaces of solutions to fractional heat equations, \textit{J. Funct. Anal.} \textbf{274} (2018), no.9, 2634-2660.


\bibitem{han2021regularity}
B.S. Han, A regularity theory for stochastic partial differential equations driven by multiplicative space-time white noise with the random fractional Laplacians, \textit{Stoch. Partial Differ. Equ. Anal. Comput.} \textbf{9} (2021), no.4, 940-983.


\bibitem{hoh1996dirichlet}
W. Hoh, N. Jacob, On the Dirichlet problem for pseudodifferential operators generating Feller semigroups, \textit{J. Funct. Anal.} \textbf{137} (1996), no.1, 19-48.


\bibitem{kim2013parabolic}
I. Kim, K.H. Kim, P. Kim, Parabolic Littlewood-Paley inequality for $\phi(-\Delta)$-type operators and applications to stochastic integro-differential equations, \textit{Adv. Math.} \textbf{249} (2013), 161–203.

\bibitem{kim2015holder}
I. Kim, K.H. Kim, A H\"older regularity theory for a class of non-local elliptic equations related to subordinate Brownian motions, \textit{Potential Anal.} \textbf{43} (2015), no.4, 653-673.

\bibitem{kim2016lp}
I. Kim, K.H. Kim, An $L_p$-theory for a class of non-local elliptic equations related to nonsymmetric measurable kernels, \textit{J. Math. Anal. Appl.} \textbf{434} (2016), no.2, 1302-1335.

\bibitem{kim2019lp}
I. Kim, K.H. Kim, P. Kim, An $L_p$-theory for diffusion equations related to stochastic processes with non-stationary independent increment, \textit{Trans. Amer. Math. Soc.} \textbf{371} (2019), no.5, 3417-3450.

\bibitem{KK2004}
K.H. Kim, N.V. Krylov, On the Sobolev space theory of parabolic and elliptic equations in $C^1$ domains, \textit{SIAM J. Math. Anal.} \textbf{36} (2004), no.2, 618-642.

\bibitem{kim2021lq}
K.H. Kim, D. Park, J. Ryu, An $L_q (L_p)$-theory for diffusion equations with space-time nonlocal operators, \textit{J. Differential Equations} \textbf{287} (2021), 376-427.

\bibitem{kim2019boundary}
M. Kim, P. Kim, J. Lee, K.A. Lee, Boundary regularity for nonlocal operators with kernels of variable orders, \textit{J. Funct. Anal.} \textbf{277} (2019), no.1, 279-332.

\bibitem{kim2021generalized}
M. Kim, K.A. Lee, Generalized Evans–Krylov and Schauder type estimates for nonlocal fully nonlinear equations with rough kernels of variable orders, \textit{J. Differential Equations} \textbf{270} (2021), 883-915.

\bibitem{kry99analytic}
N.V. Krylov, An analytic approach to SPDEs in \textit{Stochastic Partial Differential Equations: Six perspectives}, Mathematical Surveys and Monographs Vol. 64, American Mathematical Society, Providence, 1999, 185–242.

\bibitem{Krysome19}
N.V. Krylov, Some properties of weighted Sobolev spaces in $\mathbb {R}^d_+$, \textit{Ann Scuola Norm. Sup. Pisa Cl. Sci.} \textbf{28} (1999), no.4, 675-693.

\bibitem{kry99weighted}
N.V. Krylov, Weighted Sobolev spaces and Laplace's equation and the heat equations in a half space, \textit{Comm. Partial Differential Equations} \textbf{24} (1999), no.9-10, 1611-1653.


\bibitem{krylov2001some}
N.V. Krylov, Some properties of traces for stochastic and deterministic parabolic weighted Sobolev spaces, \textit{J. Funct. Anal.} \textbf{183} (2001), no.1, 1-41.

\bibitem{kuhn2019schauder}
F. K\"uhn, Schauder estimates for equations associated with L\'evy generators, \textit{Integral Equations Operator Theory} \textbf{91} (2019), no.2, 1-21.

\bibitem{kuhn2021interior}
F. K\"uhn, Interior Schauder estimates for elliptic equations associated with L\'evy operators, \textit{Potential Anal.} \textbf{56} (2022), no.3, 459-481.

\bibitem{leonori2015basic}
T. Leonori, I. Peral, A. Primo, F. Soria, Basic estimates for solutions of a class of nonlocal elliptic and parabolic equations, \textit{Discrete Contin. Dyn. Syst.} \textbf{35} (2015), no.12, 6031-6068.

\bibitem{lototsky2000sobolev}
S.V. Lototsky, Sobolev spaces with weights in domains and boundary value problems for degenerate elliptic equations, \textit{Methods Appl. Anal.} \textbf{7} (2000), no.1, 195-204.

\bibitem{mikulevivcius2017p}
R. Mikulevi\v{c}ius, C. Phonsom, On $L^p-$theory for parabolic and elliptic integro-differential equations with scalable operators in the whole space, \textit{Stoch. Partial Differ. Equ. Anal. Comput.} \textbf{5} (2017), no.4, 472–519.

\bibitem{mikulevivcius2019cauchy}
R. Mikulevi\v{c}ius, C. Phonsom, On the Cauchy problem for integro-differential equations in the scale of spaces of generalized smoothness, \textit{Potential Anal.} \textbf{50} (2019), no.3, 467-519.

\bibitem{nowak2020hs}
S. Nowak, $H^{s,p}$ regularity theory for a class of nonlocal elliptic equations, \textit{Nonlinear Anal.} \textbf{195} (2020), Article 111730.

\bibitem{ros2014dirichlet}
X. Ros-Oton, J. Serra, The Dirichlet problem for the fractional Laplacian: regularity up to the boundary, \textit{J. Math. Pures Appl.} \textbf{101} (2014), no.3, 275-302.

\bibitem{ros2016boundary}
X. Ros-Oton,  J. Serra, Boundary regularity for fully nonlinear integro-differential equations, \textit{Duke Math. J.} \textbf{165} (2016), no.11, 2079-2154.

\bibitem{ros2016regularity}
X. Ros-Oton, J. Serra, Regularity theory for general stable operators, \textit{J. Differential Equations} \textbf{260} (2016), no.12, 8675-8715.

\bibitem{ros2016dirichlet}
X. Ros-Oton, E. Valdinoci, The Dirichlet problem for nonlocal operators with singular kernels: convex and nonconvex domains, \textit{Adv. Math.} \textbf{288} (2016), 732-790.

\bibitem{triebel2010theory}
H. Triebel, \textit{Theory of function spaces}, Modern Birkh\"auser Classics, Birkh\"auser/Springer Basel AG, Basel, 2010.

\bibitem{zhang2013lp}
X. Zhang, $L^p$-solvability of nonlocal parabolic equations with spatial dependent and non-smooth kernels, preprint, 2012, arXiv:1206.2709.

\bibitem{zhang2018dirichlet}
X. Zhang, G. Zhao, Dirichlet problem for supercritical nonlocal operators, preprint, 2018, arXiv:1809.05712.

\end{thebibliography}
\end{document}